\documentclass[11pt]{amsart}
\usepackage[a4paper, margin=2cm, top=2.2cm, headsep=0.75cm, bottom=1.8cm, footskip=0.9cm]{geometry}             
\usepackage[dvipsnames]{xcolor}   		             		
\usepackage{graphicx}
\usepackage{amssymb,mathrsfs}
\usepackage{amsthm,amsmath,stmaryrd}
\usepackage{tikz}
\usepackage{tikz-cd}
\usetikzlibrary{shapes, shapes.geometric, calc, decorations.markings}
\usepackage{upgreek, enumitem}
\usepackage{bm}
\usepackage{mathtools}
\usepackage[all]{xy}
\usepackage{caption}
\usepackage{array,booktabs}

\DeclareMathAlphabet{\mathpzc}{OT1}{pzc}{m}{it}

\usepackage{hyperref}
\hypersetup{
  colorlinks   = true,          
  urlcolor     = blue,          
  linkcolor    = purple,          
  citecolor   = blue             
}

\usepackage[capitalize]{cleveref}

\newtheorem{mainthm}{Theorem}
\crefname{mainthm}{Theorem}{Theorems}
\newtheorem{maincor}[mainthm]{Corollary}
\crefname{maincor}{Corollary}{Corollaries}

\crefname{phenomenon}{Phenomenon}{Phenomena}

\newtheorem{theorem}{Theorem}[section]
\newtheorem{corollary}[theorem]{Corollary}
\newtheorem{lemma}[theorem]{Lemma}
\newtheorem{proposition}[theorem]{Proposition}

\theoremstyle{definition}
\newtheorem{definition}[theorem]{Definition}
\newtheorem{example}[theorem]{Example}

\theoremstyle{remark}
\newtheorem{rem1}[theorem]{Remark}
\newenvironment{remark}{\begin{rem1}\em}{\end{rem1}}

\title[Quantum cohomology and Fukaya summands from monotone Lagrangian tori]{Quantum cohomology and Fukaya summands \\ from monotone Lagrangian tori}

\author{Jack Smith}
\address{St John's College, Cambridge, CB2 1TP, United Kingdom}
\email{\href{mailto:j.smith@dpmms.cam.ac.uk}{j.smith@dpmms.cam.ac.uk}}

\renewcommand{\phi}{\varphi}
\newcommand{\eps}{\varepsilon}
\newcommand{\lb}{\left(}
\newcommand{\rb}{\right)}
\newcommand{\id}{\operatorname{id}}

\newcommand{\CC}{\mathbb{C}}
\newcommand{\RR}{\mathbb{R}}
\newcommand{\ZZ}{\mathbb{Z}}
\newcommand{\QQ}{\mathbb{Q}}
\newcommand{\PP}{\mathbb{P}}
\newcommand{\LL}{L^\flat}
\newcommand{\hatS}{\widehat{S}}

\newcommand{\hatCO}{\widehat{\CO}{}^0_\bL}

\newcommand{\hatB}{\widehat{\calB}}
\newcommand{\hatE}{\widehat{\calE}}
\newcommand{\hatPhi}{\widehat{\Phi}}
\newcommand{\hatTheta}{\widehat{\Theta}_{\LL}}

\newcommand{\calA}{\mathcal{A}}
\newcommand{\calB}{\mathcal{B}}
\newcommand{\calC}{\mathcal{C}}
\newcommand{\calD}{\mathcal{D}}
\newcommand{\calE}{\mathcal{E}}

\newcommand{\calG}{\mathcal{G}}
\newcommand{\calL}{\mathcal{L}}

\newcommand{\calQ}{\mathcal{Q}}

\newcommand{\calT}{\mathcal{T}}

\setenumerate{label=(\roman*)}

\newcommand{\kk}{\mathbf{k}}
\newcommand{\Char}{\operatorname{char}}

\newcommand{\m}{\mathfrak{m}}

\renewcommand{\mod}{\mathbin{\mathrm{mod}}}

\newcommand{\bL}{\mathbf{L}}

\newcommand{\op}{\mathrm{op}}

\newcommand{\mf}{\mathrm{mf}}

\newcommand{\Cl}{\operatorname{\mathnormal{C\kern-0.15ex \ell}}}
\newcommand{\diff}{\mathrm{d}}
\newcommand{\rH}{\operatorname{H}}

\newcommand{\HH}{\operatorname{HH}}
\newcommand{\rCC}{\operatorname{CC}}

\newcommand{\HF}{\operatorname{HF}}
\newcommand{\Fuk}{\mathcal{F}}

\newcommand{\CF}{\operatorname{CF}}

\newcommand{\Spec}{\operatorname{Spec}}

\newcommand{\LMF}[1][]{\operatorname{\mathcal{LM}}_{#1}}

\newcommand{\End}{\operatorname{End}}

\newcommand{\CO}{\operatorname{\mathcal{CO}}}

\newcommand{\OC}{\operatorname{\mathcal{OC}}}
\newcommand{\QH}{\operatorname{QH}}
\newcommand{\SH}{\operatorname{SH}}
\newcommand{\Jac}{\operatorname{Jac}}
\newcommand{\isol}{\mathrm{isol}}
\newcommand{\Jacisol}{\operatorname{Jac}_\isol}
\newcommand{\COisol}{\CO^0_{\bL, \isol}}
\newcommand{\qprod}{\mathbin{\star}}
\newcommand{\xto}[1]{\xrightarrow{\ \ #1 \ \ }}

\newcommand{\Shk}[2]{\langle #1, #2 \rangle_\mathrm{Shk}}
\newcommand{\Poinc}[2]{\langle #1, #2 \rangle_X}
\newcommand{\Tw}{\operatorname{Tw}}
\newcommand{\oPi}{\operatorname{\Pi}}
\newcommand{\PiTw}{\oPi \Tw}

\begin{document}

\begin{abstract}
Let $L$ be a monotone Lagrangian torus inside a compact symplectic manifold $X$, with superpotential $W_L$.  We show that a geometrically-defined closed--open map induces a decomposition of the quantum cohomology $\QH^*(X)$ into a product, where one factor is the localisation of the Jacobian ring $\Jac W_L$ at the set of isolated critical points of $W_L$.  The proof involves describing the summands of the Fukaya category corresponding to this factor---verifying the expectations of mirror symmetry---and establishing an automatic generation criterion in the style of Ganatra and Sanda, which may be of independent interest.  We apply our results to understanding the structure of quantum cohomology and to constraining the possible superpotentials of monotone tori.
\end{abstract}

\maketitle


\section{Introduction}
\label{secIntroduction}


\subsection{Statement of results}
\label{sscStatement}

Let $(X, \omega)$ be a compact symplectic manifold of dimension $2n$, and let $L \subset X$ be a monotone Lagrangian torus.  The goal of this paper is to explore the relationship between the quantum cohomology of $X$ and a Laurent polynomial---the \emph{superpotential}---associated to $L$.

Recall that monotonicity of $L$ means the existence of some $\tau > 0$ such that the Maslov index and area homomorphisms $\mu, \omega : \pi_2(X, L) \to \RR$ satisfy $\omega = \tau\mu$.  This is the relative analogue of being Fano, and allows us to construct the quantum cohomology $\QH^*(X)$ and the Floer cohomology $\HF^*(L, L)$ of $L$ with itself, over an arbitrary ground field $\kk$, using only classical transversality techniques.

These cohomology groups are connected by a $\ZZ/2$-graded unital $\kk$-algebra homomorphism
\[
\CO^0_L : \QH^*(X) \to \HF^*(L, L),
\]
which is defined on an arbitrary class $\alpha \in \QH^*(X)$ roughly as follows.  Pick a Poincar\'e dual cycle $Z \subset X$, and consider pseudoholomorphic discs in $X$ with boundary on $L$ which map an interior marked point to $Z$.  Look at the pseudocycle on $L$ swept out by a boundary marked point on such discs, and take its Poincar\'e dual class in $\HF^*(L, L)$.  This is the \emph{length-zero closed--open string map}.

More generally, one can equip $L$ with a rank-$1$ local system $\calL$ over $\kk$, and consider the same constructions for the pair $\LL = (L, \calL)$ instead of $L$.  Incorporating $\calL$ here means that every count of pseudoholomorphic discs---in the definition of the Floer differential, the Floer product, and the closed--open string map---is weighted by the monodromy of $\calL$ around each disc boundary.  Different choices of $\calL$ allow us to extract different information.  For example, $\HF^*(\LL, \LL)$ typically vanishes for most choices of $\calL$, but is non-zero when $\calL$ takes specific values.  And the map $\CO^0_{\LL}$ usually `sees' different parts of $\QH^*(X)$ for different choices of $\calL$.

\begin{example}
If $\kk = \CC$ and $L$ is the Clifford torus in $X = \CC\PP^n$, defined by
\[
L = \{[x_0 : \dots : x_n] : \lvert x_0 \rvert = \dots = \lvert x_n \rvert\},
\]
then by \cite{Cho} we have $\HF^*(\LL, \LL) \neq 0$ for exactly $n+1$ choices of $\calL$, namely those which are invariant under permutations of the homogeneous coordinates.  Moreover, if $H \in \rH^2(X)$ denotes the hyperplane class then we have
\[
\QH^*(X) = \CC[H] / (H^{n+1} - 1) \cong \prod_{\mathclap{\substack{\zeta \text{ such that}\\ \zeta^{n+1} = 1}}} \CC[H] / (H - \zeta).
\]
The projection maps onto the $n+1$ factors of this splitting can be realised by the maps $\CO^0_{\LL}$ for the $n+1$ interesting choices of $\calL$.
\end{example}

For a general monotone torus $L$, its Floer theory is controlled by its superpotential
\[
W_L \in \ZZ[\rH_1(L; \ZZ)].
\]
This counts rigid pseudoholomorphic discs $u$ which send a boundary marked point to a chosen point $p$ in $L$, weighted by the monomial $z^{[\partial u]}$ associated to the boundary homology class, and by a sign which we will not dwell on.  A standard cobordism argument shows that the resulting $W_L$ is independent of auxiliary data, if these are chosen suitably generically.  We will view $W_L$ as an element of $S = \kk[\rH_1(L; \ZZ)]$, and hence as a function on the space $\rH^1(L; \kk^\times) = \Spec S$ of rank-$1$ local systems on $L$ over $\kk$.  It is well-known \cite[Section 13]{ChoOh}, \cite[Proposition 3.3.1]{BiranCorneaLTEG} that $\HF^*(\LL, \LL)$ is non-zero if and only if $\calL$ is a critical point of $W_L$ when viewed in this way.  Explicitly, choosing a basis $\gamma_1, \dots, \gamma_n$ for $\rH_1(L; \ZZ)$ gives an identification $S \cong \kk[z_1^{\pm 1}, \dots, z_n^{\pm 1}]$, where $z_i = z^{\gamma_i}$, under which $W_L$ becomes a Laurent polynomial in the $z_i$ and the local system $\calL$ corresponds to the point whose $z_i$-coordinate is the monodromy of $\calL$ around $\gamma_i$.  The critical point condition is then that all partial derivatives $\partial W_L / \partial z_i$ vanish at $\calL$.

To extract the information of all $\calL$ at once, we can work over $S$ instead of $\kk$, and weight each disc $u$ by $z^{[\partial u]}$ instead of by the monodromy of $\calL$ around $\partial u$.  We denote the resulting Floer cohomology algebra and closed--open string map by
\begin{equation}
\label{eqCObL}
\CO^0_\bL : \QH^*(X) \to \HF^*_S(\bL, \bL).
\end{equation}
By the same calculations from \cite{BiranCorneaLTEG}, the $S$-subalgebra of $\HF^*_S(\bL, \bL)$ generated by the unit is the Jacobian ring $\Jac W_L = S / (\partial W / \partial z_1, \dots, \partial W / \partial z_n)$.  For each $\calL$, corresponding to a maximal ideal $\m$ in $S$ with residue field $\kk$, the map $\CO^0_{\LL}$ can be recovered from $\CO^0_\bL$ by reducing modulo $\m$ at chain level.

We call a critical point $\calL$ \emph{isolated} if any of the equivalent conditions in \cref{lemIsolatedTFAE} holds, and it is these $\calL$ that are relevant for our purposes.  Let $S_\isol$ denote the localisation of $S$ at the set of isolated critical points, i.e.~the ring obtained from $S$ by adjoining a multiplicative inverse to any function which is non-zero at every such $\calL$.  Similarly, let $\Jacisol W_L$ and $\CO^0_{\bL, \isol} : \QH^*(X) \to \HF^*_\isol(\bL, \bL)$ denote the localisations of $\Jac W_L$ and \eqref{eqCObL} respectively at the set of isolated critical points.  Equivalently one could apply the operation ${-} \otimes_S S_\isol$ to $\Jac W_L$ and \eqref{eqCObL}.  Since localisation is exact, it doesn't matter whether we localise at chain level or after passing to cohomology.  Our main result is the following.

\begin{mainthm}[\cref{corThmA}]
\label{TheoremA}
As an $S_\isol$-algebra, $\HF^*_\isol(\bL, \bL)$ is generated by the unit and is equal to $\Jacisol W_L$.  The $\ZZ/2$-graded unital $\kk$-algebra homomorphism
\[
\CO^0_{\bL, \isol} : \QH^*(X) \to \HF^*_\isol(\bL, \bL) = \Jacisol W_L
\]
is surjective, and gives rise to an orthogonal decomposition
\[
\QH^*(X) = K \times K^\perp
\]
(of $\ZZ/2$-graded unital $\kk$-algebras) with respect to the Poincar\'e pairing.  Here $K = \ker \CO^0_{\bL, \isol}$ and $\CO^0_{\bL, \isol}$ induces an isomorphism $K^\perp \to \Jacisol W_L$.
\end{mainthm}

\begin{maincor}
\label{CorollaryB}
$\QH^*(X)$ has $\Jacisol W_L$ as a factor, in even degree.\hfill$\qed$
\end{maincor}

\begin{remark}
This result has a natural interpretation in terms of mirror symmetry, which also explains why it is necessary to remove non-isolated critical points; see \cref{sscMS}.
\end{remark}

\begin{remark}
Sanda \cite[Section 5.1]{Sanda} proves a version of \cref{CorollaryB} but for the localisation of $\Jac W_L$ at the \emph{non-degenerate} (i.e.~Morse) critical points.  In this case, the resulting localisation is simply a product of copies of $\kk$, one for each non-degenerate critical point.  Our result strengthens this to allow degenerate isolated critical points, corresponding to more complicated factors of $\QH^*(X)$.
\end{remark}

Because $\HF^*_\isol(\bL, \bL)$ is generated by the unit, which is Poincar\'e dual to any point $p$ in $L$, the map $\CO^0_{\bL, \isol}$ in \cref{TheoremA} has a particularly nice geometric description: given an input $\alpha \in \QH^*(X)$, pick a Poincar\'e dual cycle $Z$ as above, and count rigid pseudoholomorphic discs which send an interior marked point to $Z$ and a boundary marked point to $p$, weighted by boundary class monomial in $S_\isol$.  As usual, one must also weight by a sign, and use suitably generic auxiliary data.  If $Z$ has (real) codimension-$2$ and is disjoint from $L$ then it defines a Poincar\'e dual class $\widetilde{\alpha} \in \rH^2(X, L; \kk)$ and $\CO^0_{\bL, \isol}(\alpha)$ can be computed by counting the same discs $u$ as $W_L$, but additionally weighted by the pairing $\langle [u], \widetilde{\alpha} \rangle$.  In particular, using the well-known argument of Auroux--Kontsevich--Seidel \cite[Lemma 6.7]{AurouxTDuality}, we see that $\CO^0_{\bL, \isol}(c_1(X))$ is the image of $W_L$ in $\Jacisol W_L$.

\begin{example}
Returning to the Clifford torus in $\CC\PP^n$, for a suitable choice of basis of $\rH_1(L; \ZZ)$ we have
\[
W_L = z_1 + \dots + z_n + \frac{1}{z_1 \cdots z_n}.
\]
We therefore have, over any field $\kk$, that
\[
\Jac W_L = \kk[z_1^{\pm 1}, \dots, z_n^{\pm 1}] \Big/ \Big(1 - \frac{1}{z_1^2 z_2 \cdots z_n}, \dots, 1 - \frac{1}{z_1 \cdots z_{n-1} z_n^2} \Big) \cong \kk[z] / (z^{n+1} - 1),
\]
with each $z_i$ sent to $z$.  All critical points are isolated, so $\Jacisol W_L = \Jac W_L$, and we have
\[
\CO^0_{\bL, \isol}((n+1)H) = \CO^0_{\bL, \isol}(c_1(X)) = W_L = z + \dots + z + \frac{1}{z^n} = (n+1)z.
\]
Dividing through by $n+1$ then gives $\CO^0_{\bL, \isol}(H) = z$.  (If the characteristic of $\kk$ divides $n+1$ then we can make the argument over $\ZZ$ instead of $\kk$, divide through by $n+1$ there, and only then tensor with $\kk$.  Alternatively we can compute $\CO^0_{\bL, \isol}(H)$ directly, by representing $H$ by a coordinate hyperplane $Z = \{x_i = 0\}$ and obtaining $\CO^0_{\bL, \isol}(H) = z_i = z$.)  We conclude that the map
\[
\CO^0_{\bL, \isol} : \QH^*(X) = \kk[H]/(H^{n+1} - 1) \to \HF^*_\isol(\bL, \bL) = \kk[z]/(z^{n+1} - 1)
\]
is an isomorphism, and induces the trivial splitting of $\QH^*(X)$.  This is true more generally when $L$ is a monotone toric fibre \cite{FOOOMirror}.
\end{example}

Perhaps surprisingly, the proof of \cref{TheoremA} goes via the (monotone) Fukaya category, as follows.  We first modify Sanda's proof of \cite[Theorem 1.1]{Sanda} to establish an automatic split-generation result, \cref{thmGen}, without assuming existence of a cyclic structure on the category.  This may be of independent interest.  Ganatra states a similar (more general) result, without proof, in \cite[Remark 35]{GanatraAutomaticGeneration}; see also Ganatra--Perutz--Sheridan \cite[Theorem 5.2]{GanatraPerutzSheridan}.  We then apply this in \cref{sscLMF}, using a description of the $A_\infty$-structure on $\CF^*(\LL, \LL)$ from \cite{SmithSuperfiltered} and a smoothness result of Dyckerhoff \cite[Section 7]{Dyckerhoff}, to show that for each isolated critical point $\calL_i$ the object $\LL_i = (L, \calL_i)$ split-generates a summand of the Fukaya category and has an associated factor $Q_i$ in $\QH^*(X)$.  After arguing in \cref{sscCombining} that the different $Q_i$ are pairwise orthogonal, in \cref{sscCOL} we use results from \cite{SmithHHviaHF} to interpret the resulting splitting of $\QH^*(X)$ in terms of $\CO^0_{\bL, \isol}$.

More precisely, recall that for each $\lambda \in \kk$ Sheridan \cite{SheridanFano} constructs a $\ZZ/2$-graded $\kk$-linear monotone Fukaya category $\Fuk(X)_\lambda$.  We will assume that the category has been completed with respect to cones and summands, i.e.~that we have passed to the the split-closure of its pre-triangulated envelope, in the language of \cite{SeidelBook}; see \cref{sscGeneration} for more details.  For each choice of $\calL$, the pair $\LL = (L, \calL)$ defines an object in $\Fuk(X)_\lambda$ for $\lambda = W_L(\calL)$, and by the above discussion this object is non-zero if and only if $\calL$ is a critical point of $W_L$.  Given a factor $Q$ of $\QH^*(X)$, there is an associated full subcategory $\Fuk(X)_{\lambda, Q}$ of $\Fuk(X)_\lambda$, comprising those objects $K$ for which $\CO^0_K(Q)$ contains the unit in $\HF^*(K, K)$; see \cite[Appendix A]{SmithHHviaHF} for an exposition of this notion.  The formal version of the previous paragraph would then say that $\LL_i$ is contained in and split-generates $\Fuk(X)_{\lambda_i, Q_i}$, where $\lambda_i = W_L(\calL_i)$.

\begin{remark}
\label{rmkArtinian}
Since $\QH^*(X)$ is Artinian (because it's finite-dimensional over $\kk$), it can be decomposed as a product of Artinian local rings $\calQ_1 \times \dots \times \calQ_m$, which can be decomposed no further.  The categories corresponding to different $\calQ_j$ are orthogonal, so to understand $\Fuk(X)_{\lambda}$ one may as well understand each of the summands $\Fuk(X)_{\lambda, \calQ_j}$.  In our case, each of the factors $Q_i$ is local so coincides with one of the $\calQ_j$.

If we assume that $\Char \kk \neq 2$, or if we require the Lagrangians in our Fukaya category to be orientable, then we can drop the subscript $\lambda$ from our notation, as follows.  Consider the $\kk$-linear endomorphism of $\QH^*(X)$ given by quantum multiplication by $c_1(X)$.  We denote this by $c_1 \qprod$.  Each $\calQ_j$ is contained in a generalised eigenspace of $c_1 \qprod$, say with eigenvalue $\nu_j$, otherwise it would decompose as a product of its intersections with different generalised eigenspaces, contradicting locality.  We then have that $\Fuk(X)_{\lambda, \calQ_j} = 0$ unless $\lambda = \nu_j$; see \cite[Appendix A]{SmithHHviaHF} for further discussion.  We could therefore abbreviate $\Fuk(X)_{\nu_j, \calQ_j}$ to $\Fuk(X)_{\calQ_j}$, and write
\[
\Fuk(X)_\lambda = \bigoplus_{\mathclap{\substack{j \text{ such that} \\ \nu_j = \lambda}}} \Fuk(X)_{\calQ_j}.
\]
\end{remark}

As a by-product of our arguments we actually give a complete description of the categories arising.

\begin{mainthm}[\cref{propCategoryDescriptions}]
\label{TheoremC}
For each $i$, the category $\Fuk(X)_{\lambda_i, Q_i}$ split-generated by $\LL_i$ is quasi-equivalent, as a $\ZZ/2$-graded $A_\infty$-category over $\kk$, to $\mf(\hatS, W_L - \lambda_i)$.  This is the dg-category of matrix factorisations of $W_L - \lambda_i$ over $\hatS$, which is defined to be the completion of $S$ at $\calL_i$.  Under this quasi-equivalence, $\LL_i$ is sent to the matrix factorisation corresponding to the skyscraper sheaf at the unique closed point of $\Spec \hatS$.
\end{mainthm}

This is again consistent with the mirror symmetry picture described in \cref{sscMS}.

\begin{definition}
\label{defToroidal}
A \emph{toroidal subcategory} is a summand $\Fuk(X)_{\lambda, Q}$ of the monotone Fukaya category that is split-generated by a monotone Lagrangian torus equipped with an isolated critical point of its superpotential, which we call a \emph{toroidal generator}.  So in our notation the toroidal subcategories are the $\Fuk(X)_{\lambda_i, Q_i}$, and the $\LL_i$ are toroidal generators.
\end{definition}

\Cref{TheoremC} leads to the following automatic split-generation result for toroidal subcategories, which is of a somewhat different flavour from \cref{thmGen}.

\begin{mainthm}[\cref{propSegal}]
\label{TheoremD}
Each toroidal subcategory is split-generated by any non-zero object.
\end{mainthm}

A similar phenomenon was observed by Sheridan \cite[Corollary 2.19]{SheridanFano} for summands $\Fuk(X)_{\lambda, Q}$ in the case where $Q$ is $1$-dimensional over $\kk$, and by Evans--Lekili \cite[Remark 1.2.3]{EvansLekiliGeneration} in the context of subcategories split-generated by homogeneous Lagrangians.  As noted in \cite[p158]{EvansLekiliGeneration}, one can view this as a form of $A_\infty$- or categorical-semisimplicity.

We conjecture that any two toroidal generators of a toroidal subcategory are quasi-isomorphic, up to a shift.  We have not been able to prove this, but we have the following partial result.

\begin{mainthm}[\cref{propToriIsomc}]
\label{TheoremE}
For any two toroidal generators $T_1$ and $T_2$ of a given toroidal subcategory, we have a quasi-isomorphism
\[
T_1^{\oplus 2^{n-1}} \oplus T_1[1]^{\oplus 2^{n-1}} \simeq T_2^{\oplus 2^{n-1}} \oplus T_2[1]^{\oplus 2^{n-1}}.
\]
\end{mainthm}

\begin{remark}
The potentially-mysterious form of \cref{TheoremE} arises from tensoring one of the $T_i$ with a Koszul complex, which has $2^{n-1}$ terms in degree $0 \mod 2$ and $2^{n-1}$ terms in degree $1 \mod 2$.
\end{remark}

Before further discussion and application of our results, we make a few technical comments.

\begin{remark}
\label{rmkFukTechnical}
\begin{enumerate}
\item Each Lagrangian in the Fukaya category should be equipped with a $\ZZ/2$-grading with respect to a fixed $\ZZ/2$-grading of $X$, in the sense of \cite{SeidelGraded}.  If we only care about consequences for $\QH^*(X)$ then we can equip $X$ with its canonical $\ZZ/2$-grading, in which case a $\ZZ/2$-grading of a Lagrangian is just an orientation (in particular, the torus $L$ has a $\ZZ/2$-grading).
\item Each Lagrangian should also be equipped with a relative spin structure, to determine the signs of pseudoholomorphic disc counts.  In particular, the superpotential $W_L$ of the torus $L$ depends on the choice of relative spin structure on $L$.  Happily there is a canonical choice in this case, namely the unique spin structure induced by picking an arbitrary identification $L \cong \RR^n / \ZZ^n$ and using this to trivialise $TL$.
\item To do monotone Floer theory with multiple Lagrangians, we need to assume that the fundamental group of each has trivial image in $\pi_1(X)$.  For all statements involving the Fukaya category, therefore, we should impose the additional hypothesis on each Lagrangian $K$ that $\pi_1(K)$ has trivial image in $\pi_1(X)$.  (Alternatively one could weaken these conditions to be about $\rH_1$ instead of $\pi_1$, at the expense of strengthening the monotonicity condition to be in terms of relative $\rH_2$ rather than relative $\pi_2$.)  \cref{TheoremA}, and therefore also \cref{CorollaryB}, does not require this hypothesis since it only relies on Floer theory of $L$ with itself.
\item We can twist $\QH^*(X)$ and the Fukaya category by any homomorphism $B : \rH_2(X; \ZZ) \to \kk^\times$, and our results have an analogue where the local system on each Lagrangian $K$ is replaced by a lift of $B$ to $\rH_2(X, K; \ZZ) \to \kk^\times$.  We need to assume that such a lift exists for our torus $L$, i.e.~that that restriction of $B$ to the image of $\rH_2(L; \ZZ)$ is trivial.  We leave the details to the interested reader.
\end{enumerate}
\end{remark}

\begin{remark}
One could attempt to extend our results---in particular \cref{CorollaryB}---to non-monotone tori, but this would run into several difficulties:
\begin{enumerate}
\item Suitable technical foundations for the Fukaya category are needed.  These should include establishing the properties described in \cref{sscStringMaps,sscwpCY}, except for \cref{propEspace} which is not expected to hold in general.  For simplicity we will also assume that the category constructed is strictly unital.
\item The appropriate version of the superpotential is now somewhat different.  We should work over a Novikov field $\Lambda$ of characteristic $0$, with positive-filtration part $\Lambda_{>0}$, and assume that $\rH^1(L; \Lambda_{>0})$ embeds in the space of weak bounding cochains for $\LL = (L, \calL)$.  This means that there exists a function $\mathfrak{PO} : \rH^1(L; \Lambda_{>0}) \to \Lambda_{>0}$ such that for all $b \in \rH^1(L; \Lambda_{>0})$ we have
\[
\sum_{k \geq 0} \mu^k(b, \dots, b) = \mathfrak{PO}(b) \cdot e_{\LL},
\]
where the $\mu^k$ are the $A_\infty$-operations on $\CF^*(\LL, \LL)$ and $e_{\LL}$ is the strict unit.  This function $\mathfrak{PO}$ takes the place of the formal expansion of $W_L$ about $\calL$.
\item It then makes sense to talk about `isolated critical points of $W_L$', meaning those $\calL$ for which $0$ is an isolated critical point of $\mathfrak{PO}$.  For such $\calL$ we have that $\HF^*(\LL, \LL)$ is non-zero, and given by the Clifford algebra corresponding to the Hessian of $\mathfrak{PO}$ at $0$.  However, the crucial input to our automatic split-generation result, \cref{thmGen}, is that the $A_\infty$-algebra $\calA = \CF^*(\LL, \LL)$ is homologically smooth.  In the monotone case we establish this by relating $\calA$ to the endomorphism algebra of a matrix factorisation using results from \cite{SmithSuperfiltered}, which rely on $\calA$ being a \emph{superfiltered} deformation of the exterior algebra.  Essentially this means that $\calA$ looks like $\rH^*(L)$ but with correction terms added to its $A_\infty$-operations which respect the $\ZZ/2$-grading and decrease degree with respect to the $\ZZ$-grading.  Monotonicity ensures that $\calA$ has this form.  We conjecture, based on the heuristic picture outlined in \cref{sscMS}, that the even in the non-monotone case the algebra $\CF^*(\LL, \LL)$ is smooth at isolated critical points, but one would need new techniques to establish this.
\item Because of the absence of \cref{propEspace} a new technique would be needed to show that factors $Q_i$ of $\QH^*(X)$ with different $\lambda_i$ values are orthogonal.
\end{enumerate}
\end{remark}


\subsection{Mirror symmetry picture}
\label{sscMS}

It is expected that in favourable situations $X$ is mirror to a Landau--Ginzburg model $(Y, W)$, where $Y$ is a variety over $\kk$ and $W$ is a regular function $Y \to \kk$.  Concretely, this means that the Fukaya category $\Fuk(X)_\lambda$ should be quasi-equivalent to the matrix factorisation category $\mf(Y, W-\lambda)$.  The Lagrangian torus $L$ should be mirror to to an open subvariety $U \cong \Spec S \cong (\kk^\times)^n$ of $Y$, such that $W|_U$ coincides with $W_L$.  Moreover, the object $\LL = (L, \calL)$ should be mirror to the matrix factorisation corresponding to the skyscraper sheaf at the point $\calL$ in $\Spec S$.  \cref{TheoremC} verifies this prediction at isolated critical points of $W_L$.  (Both $\LL$ and the matrix factorisation are zero if $\calL$ is not a critical point.)

The quantum cohomology of $X$ is expected (again, in favourable situations) to be isomorphic to the Hochschild cohomology of the Fukaya category.  More precisely, for each $\lambda$, the generalised $\lambda$-eigenspace of $c_1\qprod$ acting on $\QH^*(X)$ is expected to be isomorphic to $\HH^*(\Fuk(X)_\lambda)$ via the full (as opposed to length-zero) closed--open string map
\[
\CO_\lambda : \QH^*(X) \to \HH^*(\Fuk(X)_\lambda).
\]
Under mirror symmetry the right-hand side is $\HH^*(\mf(Y, W-\lambda))$, and this was computed by Lin--Pomerleano \cite[Theorem 3.1]{LinPomerleano} to be the hypercohomology of the complex $(\Lambda^* TY, [W-\lambda, \bullet])$, where $[\bullet, \bullet]$ is the Schouten--Nijenhuis bracket.  Each isolated critical point $p$ of $W$ with critical value $\lambda$ thus gives rise to a factor in $\HH^*(\mf(Y, W-\lambda))$ which looks like the local Jacobian ring of $W$ at $p$.  Focusing on those isolated critical points contained in $U$ then gives rise to \cref{CorollaryB}.

In general, $W_L$ may have non-isolated critical points.  For example, by \cite[Proposition 4.22]{PascaleffTonkonog} if $5 \leq k \leq 8$ then there is a monotone Lagrangian torus in the (monotone) $k$-point blowup of $\CC\PP^2$ whose superpotential has a $1$-dimensional critical locus.  In the notation of \cite[Table 1]{PascaleffTonkonog} this locus is given by $\{x=-1\} \cup \{y=-1\}$ for $k=5$ and by $\{x+y = -1\}$ for $k \geq 6$.  To compute the contribution to $\QH^*(X)$ from a positive-dimensional critical locus $C$ of $W_L$, mirror symmetry suggests that one has to understand the corresponding critical locus $\overline{C}$ of $W$ in $Y$.  Typically $\overline{C}$ will be a (partial) compactification of $C$, and one cannot calculate its contribution to $\QH^*(X)$ from $C$ alone.  So there is no reason to expect a simple relationship between $\QH^*(X)$ and the full Jacobian ring $\Jac W_L$.  Moreover, there is a very mundane reason why neither the localisation of $\Jac W_L$ at a non-isolated critical point nor the coordinate ring of any component of $\Spec \Jac W_L$ containing a non-isolated critical point can appear as a factor in $\QH^*(X)$: both are infinite-dimensional over $\kk$ by \cref{corNonIsolated}.

\begin{remark}
An important perspective on the picture sketched above is via the \emph{Borman--Sheridan class}.  Roughly speaking, suppose $L$ lies in the complement of an anticanonical divisor $D$, and is exact in $X \setminus D$.  Let $N \cong T^*L \subset X \setminus D$ be a Weinstein neighbourhood of $L$.  Tonkonog \cite{TonkonogSHtoLEG} shows that there exists a class $\mathcal{BS}$ in the symplectic cohomology $\SH^0(X \setminus D)$ which is sent to $1 \otimes W_L$ under the Viterbo restriction map
\[
\SH^*(X \setminus D) \to \SH^*(N) \cong \rH_{n-*}(\Lambda L) \cong \rH_{n-*}(L \times \Omega L) \cong \rH^*(L) \otimes \kk[\rH_1(L; \ZZ)].
\]
Here $\Lambda L$ and $\Omega L$ denote the free- and based-loop spaces of $L$ respectively.  On the other hand, Borman--Sheridan--Varolgunes \cite{BormanSheridanVarolgunes} and Borman--El Alami--Sheridan \cite{BormanElAlamiSheridanMaurerCartan} show that in good situations $\QH^*(X)$ can be obtained as a deformation of $\SH^*(X \setminus D)$, with deformation defined by Lie bracket with the class $\mathcal{BS}$.

The mirror symmetry interpretation of these results is that $Y = \Spec \SH^0(X \setminus D)$ is mirror to $X \setminus D$, with
\[
\HH^*(\operatorname{D^bCoh}(Y)) \cong \HH^*(\mathcal{W}(X \setminus D)) \cong \SH^*(X \setminus D),
\]
where $\mathcal{W}$ denotes the wrapped Fukaya category.  The function $W$ is then given by $\mathcal{BS}$, viewed as an element of the coordinate ring of $Y$.  The chart $U$ corresponds to $\Spec \SH^0(N)$, and Tonkonog's result corresponds to our assertion that $W|_U = W_L$.  Thinking of $W-\lambda$ as an element of $\HH^{\mathrm{even}}(\operatorname{D^bCoh}(Y))$, it defines a $\ZZ/2$-graded deformation of $\operatorname{D^bCoh}(Y)$ to $\mf(Y, W-\lambda)$, and this is mirror to deforming $\mathcal{W}(X \setminus D)$ to $\Fuk(X)_\lambda$ by compactifying $X \setminus D$ to $X$.  At the level of Hochschild cohomology, this deformation should be equivalent to turning on a differential given by `bracket with $W$', and this is what Borman, Sheridan, et al show.

It would be interesting to explore how critical points of $W_L$ are reflected in the class $\mathcal{BS}$.
\end{remark}


\subsection{Constraints on quantum cohomology from monotone tori}

\Cref{TheoremA,CorollaryB} mean that one can deduce results about the structure of $\QH^*(X)$ from the presence of a monotone torus $L \subset X$ and knowledge of $W_L$.  Sanda uses this approach in \cite{Sanda} to construct factors of $\kk$ in $\QH^*(X)$ from non-degenerate critical points of $W_L$, and from (even-dimensional, oriented, spin) Lagrangian rational homology spheres.  With our results we can now identify larger factors.

\begin{example}
\label{exQuadricQH}
The quadric threefold $X$ contains a monotone torus $L$ whose superpotential
\[
W_L = x+y+z+\frac{1}{xy}+\frac{1}{yz} \in \kk[x^{\pm1}, y^{\pm1}, z^{\pm1}] \cong S
\]
can be computed by toric degeneration \cite[Remark 7.1.3]{EvansLekiliGeneration}, \cite[Theorem 1]{NishinouNoharaUeda}.  If $\Char \kk = 2$ then $\Jac W_L$ vanishes and we don't learn anything, so assume from now on that $\Char \kk \neq 2$.  We then have
\[
\Jacisol W_L = \Jac W_L \cong \kk[y]/(y^3-4).
\]
So by \cref{CorollaryB} $\QH^*(X)$ contains a $3$-dimensional factor isomorphic to $\kk[y]/(y^3-4)$.  Since $\QH^*(X)$ is $4$-dimensional, we conclude that there is an algebra isomorphism
\[
\QH^*(X) \cong \kk[x] / (y^3 - 4) \times \kk.
\]
If $\Char \kk = 3$ then the first factor is local and does not decompose any further.

We can actually upgrade this to a complete description of $\QH^*(X)$, still assuming $\Char \kk \neq 2$, as follows.  Since the minimal Chern number of $X$ is $3$, we can equip $\QH^*(X)$ with a $\ZZ/6$-grading.  With respect to this grading $\QH^*(X)$ is $2$-dimensional in degree $0$, and $1$-dimensional in degrees $2$ and $4$, with $\QH^2(X)$ generated by $H$.  Therefore, for degree reasons, we have $H^4 = \nu H$ for some $\nu \in \kk$.  Applying $\COisol$ to this equation, we deduce that $\nu = 4$.  Therefore $\QH^*(X) = \kk[H] / (H^4 - 4H)$.
\end{example}

\begin{example}
\label{ex4Blowup}
From \cite[Proposition 4.22]{PascaleffTonkonog} the monotone blowup $X = \operatorname{Bl}_4 \CC\PP^2$ of $\CC\PP^2$ at four points contains a monotone torus with potential
\[
W_L = (1+x+y)\big( 1+\tfrac{1}{x}\big) \big( 1+ \tfrac{1}{y} \big) - 3 \in \kk[x^{\pm1}, y^{\pm1}] \cong S.
\]
Over $\CC$ this has three non-degenerate critical points, namely $(x, y) = (-1, -1)$ and $(x, y) = (\xi, \xi)$ where $\xi^2 = \xi + 1$.  Sanda \cite{Sanda} combines these with idempotents arising from Lagrangian spheres to show that $\QH^*(X)$ is semisimple in this case.  Using our results, we see that over any field $\QH^*(X)$ contains a factor of
\[
\Jacisol W_L = \Jac W_L \cong \kk \times \kk[\xi] / (\xi^2-(\xi+1)).
\]
If $\Char \kk = 5$ then the second factor is local and does not decompose further, so $\QH^*(X)$ is not semisimple.
\end{example}


\subsection{Constraints on monotone tori from quantum cohomology}

Turning \cref{CorollaryB} around, knowledge of $\QH^*(X)$ places restrictions on the possible superpotentials of monotone tori in $X$.  The crudest---but already non-trivial---constraints come from the total dimension $D$ of $\QH^*(X)$, i.e.~the sum of the Betti numbers of $X$.  The simplest statement is the following.

\begin{corollary}
The number of isolated critical points of $W_L$ is at most $D$.\hfill$\qed$
\end{corollary}

Previously this could only have been said for nondegenerate critical points.  A slightly more refined statement is the following.

\begin{corollary}
The sum of the Milnor numbers of the isolated singularities of $W_L$ is at most $D$.\hfill$\qed$
\end{corollary}

Knowledge of the ring structure on $\QH^*(X)$ of course allows one to go further.

\begin{corollary}
\label{corLocalFactors}
There is an injection from the set of isolated critical points of $W_L$ to the set of factors when $\QH^*(X)$ is decomposed as a product of local rings.  Moreover, the localisation of $\Jac W_L$ at each isolated critical point is isomorphic to the corresponding local factor of $\QH^*(X)$.\hfill$\qed$
\end{corollary}

\begin{example}
Let $L$ be any monotone Lagrangian torus in $X = \CC\PP^n$, with superpotential $W_L$.  If $\Char \kk$ does not divide $n+1$ then $\QH^*(X) \cong \kk^{\times (n+1)}$, so $W_L$ has at most $n+1$ isolated critical points, all of which are non-degenerate.  However, if $\Char \kk = p$ \emph{does} divide $n+1$, say $n+1 = p^ab$ with $b$ not divisible by $p$, then $\QH^*(X) \cong (\kk[x] / (x^{p^a}))^{\times b}$, so $W_L$ has at most $b$ isolated critical points, at each of which $\Jac W_L$ looks like $\kk[x] / (x^{p^a})$.
\end{example}

Sometimes one can argue that certain local factors of $\QH^*(X)$ cannot arise from isolated critical points of superpotentials of monotone tori, and therefore strengthen \cref{corLocalFactors}.  This idea is illustrated in \cref{exCubic,exQuadric}, where we prove the following.

\begin{proposition}
\label{corCubic}
Suppose $L$ is a monotone Lagrangian torus in the cubic surface.  If $\Char \kk \neq 3$ then $W_L$ has at most one isolated critical point and it must be non-degenerate, whilst if $\Char \kk = 3$ then $W_L$ cannot have any isolated critical points.
\end{proposition}

\begin{remark}
Pascaleff--Tonkonog \cite[Table 1]{PascaleffTonkonog} give a monotone torus $L$ in the cubic surface with
\[
W_L = \frac{(1+x+y)^3}{xy} - 6.
\]
This has a critical point at $(1, 1)$, which is non-degenerate outside characteristic $3$.  In characteristic $3$ it falls onto the $1$-dimensional critical locus $\{1+x+y = 0\}$, and thus becomes non-isolated.
\end{remark}

\begin{proposition}
\label{corQuadric}
Consider the quadric threefold $X$, and let $V$ be the Lagrangian sphere arising as the vanishing cycle of the toric degeneration (from \cref{exQuadricQH}) of $X$ to a cone over a quadric surface.  Suppose $K$ is a monotone Lagrangian torus in $X$ which is disjoint from $V$.  If $\Char \kk = 2$ then $W_K$ has no isolated critical points; if $\Char \kk = 3$ then it has at most one isolated critical point, with local Jacobian ring $\kk[y] / (y-1)^3$; and if $\Char \kk \neq 2, 3$ then it has at most three isolated critical points, all of which are non-degenerate.
\end{proposition}

\begin{remark}
The torus $L$ considered in \cref{exQuadricQH} is disjoint from $V$ and attains these bounds.
\end{remark}

\begin{remark}
Our constraints on $W_L$ are essentially local in flavour, although cumulative over the different isolated critical points.  There are also strong global constraints on $W_L$ coming from the work of Tonkonog \cite{TonkonogDescendants}, in which it is shown that the \emph{periods} of $W_L$, i.e.~the constant terms in powers of $W_L$, are independent of $L$.  It would be interesting to explore what can be done by combining these two flavours of constraint.
\end{remark}


\subsection{Isolated critical points}
\label{sscIsolated}

We end this introduction with a brief summary of the notion of isolatedness of critical points, which plays a key role in our results.

Each critical point $\calL$ of $W_L$ corresponds to a $\kk$-valued point of $\Spec \Jac W_L$, or equivalently a maximal ideal $\m \subset \Jac W_L$ with residue field $\kk$.  Let $\Jac_\calL W_L$ denote the localisation of $\Jac W_L$ at $\calL$, i.e.~at the ideal $\m$.

\begin{lemma}
\label{lemIsolatedTFAE}
For a critical point $\calL$ of $W_L$, the following are equivalent:
\begin{enumerate}
\item\label{itmFinDim} $\Jac_\calL W_L$ is finite-dimensional over $\kk$.
\item\label{itmArtin} $\Jac_\calL W_L$ is Artinian.
\item\label{itmKrullZero} $\Jac_\calL W_L$ has Krull dimension zero.
\item\label{itmIrreducible} $\calL$ constitutes an irreducible component of $\Spec \Jac W_L$.
\item\label{itmIsolated} $\calL$ is an isolated point of $\Spec \Jac W_L$, in the sense that $\{\calL\}$ is open in $\Spec \Jac W_L$.
\end{enumerate}
\end{lemma}
\begin{proof}
\ref{itmFinDim} $\iff$ \ref{itmArtin}: By \cite[Exercise 8.3]{AtiyahMacdonald} a finitely generated $\kk$-algebra is Artinian if and only if it's finite-dimensional over $\kk$.

\ref{itmArtin} $\iff$ \ref{itmKrullZero}: By \cite[Theorem 8.5]{AtiyahMacdonald} a ring is Artinian if and only if it's Noetherian and of Krull dimension zero.

\ref{itmKrullZero} $\iff$ \ref{itmIrreducible}: $\Jac_\calL W_L$ has Krull dimension zero if and only if there are no primes in $\Jac_\calL W_L$ except for the localisation of $\m$ itself.  This is equivalent to $\m$ being a minimal prime in $\Jac W_L$, and hence to $\calL$ constituting an irreducible component of $\Spec \Jac W_L$.

\ref{itmIrreducible} $\iff$ \ref{itmIsolated}: If $\calL$ constitutes an irreducible component then the union of the other components is a closed complement to $\{\calL\}$.  Conversely, if $\{\calL\}$ is open then $\{\calL\} \cup (\Spec \Jac W_L \setminus \{\calL\})$ is a decomposition of $\Spec \Jac W_L$ into disjoint closed subsets which can be refined into a decomposition into irreducible components.
\end{proof}

\begin{definition}
The critical point $\calL$ is \emph{isolated} if any of these equivalent conditions holds.
\end{definition}

\begin{corollary}
\label{corNonIsolated}
If $\calL$ is a non-isolated critical point, and $C$ is any component of $\Spec \Jac W_L$ containing $\calL$, then the localisation $\Jac_\calL W_L$ and the coordinate ring $\kk[C]$ are both infinite-dimensional over $\kk$.
\end{corollary}
\begin{proof}
The statement about $\Jac_\calL W_L$ is an immediate consequence of condition \ref{itmFinDim} above.  Suppose now, for contradiction, that $\kk[C]$ is finite-dimensional over $\kk$.  Then it's Artinian, so by \cite[Theorem 8.5]{AtiyahMacdonald} (as above) it has Krull dimension zero.  This means that the minimal prime $\mathfrak{p} \subset \Jac W_L$ defining $C$ is actually maximal, and hence equal to the ideal $\m$ defining $\calL$.  Therefore $\m$ is itself minimal, so the localisation $\Jac_\calL W_L$ of $\Jac W_L$ at $\m$ has Krull dimension zero.  This contradicts condition \ref{itmKrullZero} above, completing the proof.
\end{proof}

Note that if the isolated critical points of $W_L$ are $\calL_1, \dots, \calL_r$ (there are only finitely many, by condition \ref{itmIrreducible}) then $\Spec \Jac W_L$ decomposes into disjoint open sets
\[
\{\calL_1\} \cup \dots \cup \{\calL_r\} \cup (\Spec \Jac W_L \setminus \{\calL_1, \dots, \calL_r\}).
\]
We obtain a corresponding decomposition of $\Jac W_L$ into a direct product of rings
\[
\Jac W_L = \Jac_{\calL_1} W_L \times \dots \times \Jac_{\calL_r} W_L \times \left(\Jac W_L \big/ (\Jac_{\calL_1} W_L \times \dots \times \Jac_{\calL_r} W_L)\right).
\]

\begin{definition}
\label{defSisol}
The \emph{isolated part} of the Jacobian ring, denoted $\Jacisol W_L$, is the localisation of $\Jac W_L$ at $\{\calL_1, \dots, \calL_r\}$.  Equivalently, it is $\Jac_{\calL_1} W_L \times \dots \times \Jac_{\calL_r} W_L$.
\end{definition}

Before moving on, recall the following facts from commutative algebra, which we shall use later.

\begin{lemma}
\label{lemCommAlg}
\begin{enumerate}
\item\label{itmLocalisation} Over any ring, localisation of modules is exact \cite[Proposition 10.9.12]{stacks-project}, and is equivalent to tensoring with the localisation of the ring \cite[Lemma 10.12.15]{stacks-project}.
\item\label{itmCompletion} Over a Noetherian ring, $I$-adic completion (for any ideal $I$) of finitely-generated modules is exact \cite[Lemma 10.97.1(1)]{stacks-project} and is equivalent to tensoring with the $I$-adic completion of the ring \cite[Lemma 10.97.1(3)]{stacks-project}.\hfill$\qed$
\end{enumerate}
\end{lemma}

In particular, we can view $\Jac_\calL W_L$ (defined to be the localisation of $\Jac W_L$ at $\calL$) as the Jacobian ring of $W_L$ over the localisation $S_\calL$ of $S$ at $\calL$, or as $(\Jac W_L) \otimes_S S_\calL$.  Similarly, letting $\widehat{\phantom{-}}$ denote $\m$-adic completion, we have natural isomorphisms
\[
\widehat{\Jac W_L} \cong \text{Jacobian ring of $W_L$ over $\hatS$} \cong (\Jac W_L) \otimes_S \hatS.
\]


\subsection{Structure of the paper}

In \cref{secGeneration} we start by recalling various categorical preliminaries, in particular split-generation and the Shklyarov (or Mukai) pairing.  We then recap relevant features and properties of the Fukaya category, and finally combine these to prove a general automatic split-generation result, \cref{thmGen}, in the spirit of Sanda and Ganatra.  \Cref{secProofMain} applies these ideas to the specific case of a monotone Lagrangian torus, first considering a single isolated critical point, then showing that different isolated critical points do not interact, and finally relating the abstract categorical isomorphisms we obtain to the concrete geometrically-defined map $\COisol$.  We end by giving the details of some of the examples considered above, in \cref{secExamples}.


\subsection{Acknowledgements}

I am grateful to Ed Segal for valuable input to \cref{sscToroidal}, and in particular for explaining the proof of \cref{propSegal}.  I would also like to thank Nick Sheridan for helpful correspondence, and Ivan Smith for influential questions and comments and for drawing my attention to the example of the cubic surface.


\section{Generation from smoothness}
\label{secGeneration}


\subsection{Split-generation}
\label{sscGeneration}

First we recap various categorical preliminaries, closely following \cite{GanatraAutomaticGeneration}.

Recall that an $A_\infty$-category $\calC$ has a \emph{pretriangulated envelope}, denoted $\Tw \calC$, and an \emph{idempotent completion} or \emph{split-closure}, denoted $\oPi \calC$.  Each of these is well-defined up to quasi-equivalence and carries a cohomologically full and faithful embedding of $\calC$.  An explicit model for $\Tw \calC$ is the category of \emph{twisted complexes} (hence the name), given by embedding $\calC$ in $\calC\text{-mod}$ using Yoneda, and then taking the full subcategory obtained from the image by iteratively taking mapping cones.

We will mostly be interested in the split-closed pretriangulated envelope, $\PiTw \calC$, sometimes also denoted $\mathrm{perf}(\calC)$ (e.g.~in \cite{GanatraAutomaticGeneration}).  For example, $\Fuk(X)_\lambda$ implicitly means the split-closed pretriangulated envelope of the Fukaya category whose objects are actual Lagrangians.  The embedding of $\PiTw \calC$ into $\PiTw(\PiTw\calC)$ is a quasi-equivalence, and we say that a categorical property is \emph{Morita-invariant} if it is preserved by quasi-equivalences and by applying $\PiTw$.  Similarly, a construction is Morita-invariant if quasi-equivalences and embeddings $\calC \to \PiTw \calC$ induce (quasi-)isomorphisms on it.  For example, Hochschild homology is Morita-invariant since a functor $F : \calC \to \calD$ induces a homomorphism $\HH_*(\calC) \to \HH_*(\calD)$ which is an isomorphism if $F$ is a quasi-equivalence or an embedding of the form $\calC \to \PiTw \calC$.  Similarly, Hochschild cohomology is Morita-invariant, in the sense that a functor $F : \calC \to \calD$ induces a zigzag
\begin{equation}
\label{eqHHmaps}
\HH^*(\calC) \to \HH^*(\calC, \calD) \leftarrow \HH^*(\calD)
\end{equation}
and both of these arrows are isomorphisms if $F$ is a quasi-equivalence or an embedding $\calC \to \PiTw \calC$.

\begin{remark}
We refer the reader to \cite[Section 2]{RitterSmith} and \cite[Appendix A]{SheridanFano} for the fundamentals of Hochschild homology and cohomology in the setting of Fukaya categories.  In \eqref{eqHHmaps} the group $\HH^*(\calC, \calD)$ denotes Hochschild cohomology of $\calC$ with coefficients in $\calD$, viewed as a $\calC$-bimodule via $F$, and the groups $\HH^*(\calC)$ and $\HH^*(\calD)$ are really shorthand for $\HH^*(\calC, \calC)$ and $\HH^*(\calD, \calD)$.  The first arrow is pushforward on coefficients, whilst the second arrow is pullback on Hochschild cochains.
\end{remark}

If $\calD$ is a full subcategory of $\calC$ then there is a cohomologically full and faithful embedding of $\PiTw \calD$ into $\PiTw \calC$.  We say that $\calD$ \emph{split-generates} $\calC$ if this embedding is a quasi-equivalence.


\subsection{The Shklyarov pairing}
\label{sscShklyarov}

Recall that a dg- or $A_\infty$-category over $\kk$ is \emph{proper} if the total cohomology of each morphism space is finite-dimensional.  Any such category has a \emph{Shklyarov pairing} on its Hochschild homology, denoted $\Shk{{-}}{{-}}$.  This was introduced by Shklyarov \cite{ShklyarovHirzebruch} for dg-algebras (i.e.~dg-categories with a single object) and extended to $A_\infty$-categories by Sheridan \cite[Section 5]{SheridanHodge}, who calls it the \emph{Mukai} pairing.  Sheridan proves that this pairing is Morita-invariant \cite[Proposition 5.20]{SheridanHodge}, in the sense that it is preserved by the isomorphism on Hochschild homology induced by a quasi-equivalence or an embedding $\calC \to \PiTw \calC$.  It follows immediately from his explicit formula for the pairing \cite[(35)]{SheridanHodge} that the pairing is also preserved under inclusions of full subcategories, and that if $\calA_1, \calA_2 \subset \calC$ are full subcategories which are categorically orthogonal (morphism complexes between $\calA_1$ and $\calA_2$ are acyclic) then the images of $\HH_*(\calA_1)$ and $\HH_*(\calA_2)$ in $\HH_*(\calC)$ are orthogonal with respect to it.

Recall next that a dg- or $A_\infty$-category over $\kk$ is \emph{(homologically) smooth} if its diagonal bimodule is split-generated by Yoneda bimodules.  Shklyarov shows that if a dg-algebra is smooth and proper then its Shklyarov pairing is non-degenerate \cite[Theorem 5.3]{ShklyarovHirzebruch}, and that its Hochschild homology is finite-dimensional \cite[Theorem 4.6]{ShklyarovSerre}, so the pairing is in fact perfect.  Since smoothness is Morita-invariant \cite[Proposition 20]{GanatraAutomaticGeneration}, and every $A_\infty$-algebra is quasi-isomorphic to a dg-algebra, it follows from Sheridan's results that the same properties hold for smooth and proper $A_\infty$-algebras.

\begin{proposition}[{\cite[Proposition 5.24]{SheridanHodge}}]
\label{propSmoothProper}
If an $A_\infty$-algebra $\calA$ is smooth and proper then $\HH_*(\calA)$ is finite-dimensional and its Shklyarov pairing is perfect.\hfill$\qed$
\end{proposition}


\subsection{The closed--open and open--closed string maps}
\label{sscStringMaps}

Now we focus on the case of interest, where $\calC$ is the Fukaya category $\Fuk(X)_\lambda$.  Recall from \cite[Sections 2.5--2.6]{SheridanFano} that there is \emph{closed--open string map}
\[
\CO_\lambda : \QH^*(X) \to \HH^*(\Fuk(X)_\lambda).
\]
which is a $\ZZ/2$-graded unital $\kk$-algebra homomorphism.  There is likewise an \emph{open--closed string map}
\[
\OC_\lambda : \HH_*(\Fuk(X)_\lambda) \to \QH^{*+n}(X),
\]
which is a $\ZZ/2$-graded $\QH^*(X)$-module map, where $\alpha \in \QH^*(X)$ is defined to act on $\HH_*(X)$ via the cap product action of $\CO_\lambda(\alpha)$.  Recalling that $c_1 \qprod$ denotes the quantum multiplication action of $c_1(X)$ on $\QH^*(X)$, we have the following.

\begin{proposition}[{\cite[Theorem 9.5(1)]{RitterSmith}, or \cite[Corollary 2.11]{SheridanFano} over $\CC$}]
\label{propEspace}
Assume that $\Char \kk \neq 2$ or that we only allow orientable Lagrangians in the Fukaya category.  Then the image of $\OC_\lambda$ is contained in the generalised $\lambda$-eigenspace $\QH^*(X)_\lambda$ of $c_1 \qprod$.\hfill$\qed$
\end{proposition}

\begin{remark}
There is a corresponding result for $\CO_\lambda$, but we will not use it.  Over $\CC$ this states that $\CO_\lambda$ vanishes on $\QH^*(X)_\mu$ if $\mu \neq \lambda$; see \cite[Proposition 2.9]{SheridanFano} or \cite[Theorem 9.6]{RitterSmith}.  In general (i.e.~without assuming that $\kk$ is algebraically closed) it states that $\CO_\lambda$ vanishes on the unique $c_1 \qprod$-invariant complement $C$ to $\QH^*(X)_\lambda$ in $\QH^*(X)$.  This can be proved by factorising the characteristic polynomial $\chi(T)$ of $c_1\qprod$ as $(T-\lambda)^m p(T)$ for some non-negative integer $m$ and some polynomial $p$ with $p(\lambda) \neq 0$.  By B\'ezout's lemma there exist polynomials $f$ and $g$ such that $1 = (T-\lambda)^mf(T) + p(T)g(T)$.  The maps $(c_1 - \lambda)^mf(c_1) \qprod$ and $p(c_1)g(c_1)\qprod$ then represent projection onto $C$ and $\QH^*(X)_\lambda$ respectively.  One argues, as in \cite[Proposition 2.9]{SheridanFano}, that $\CO_{\lambda}(p(c_1))$ is invertible in $\HH^*(\Fuk(X)_\lambda)$.  On the other hand, by Cayley--Hamilton we have $\chi(c_1) = 0$ in $\QH^*(X)$, and hence
\begin{equation}
\label{eqCOespace}
0 = \CO_\lambda(\chi(c_1))= \CO_\lambda((c_1 - \lambda)^m p(c_1)) = \CO_\lambda((c_1 - \lambda)^m) \cup \CO_\lambda(p(c_1)),
\end{equation}
where $\cup$ is the cup product on $\HH^*(\Fuk(X)_\lambda)$.  Combining \eqref{eqCOespace} with invertibility of $\CO_\lambda(p(c_1))$, we deduce that $\CO_\lambda((c_1 - \lambda)^m) = 0$.  Therefore $\CO_\lambda$ annihilates all multiples of $(c_1-\lambda)^mf(c_1)$, and hence annihilates $C$.
\end{remark}

The category $\Fuk(X)_\lambda$ is proper, so its Hochschild homology comes equipped with its Shklyarov pairing.  Meanwhile, quantum cohomology carries the Poincar\'e pairing $\Poinc{{-}}{{-}}$.  Crucially, $\OC_\lambda$ is compatible with these pairings in the following sense.

\begin{proposition}[{\cite[Theorem 31, credited to Ganatra--Perutz--Sheridan]{GanatraAutomaticGeneration}}]
\label{propOCisometry}
For any classes $a$ and $b$ in $\HH_*(\Fuk(X)_\lambda)$ we have
\begin{flalign*}
&& \Shk{a}{b} &= (-1)^{n(n+1)/2}\Poinc{\OC_\lambda(a)}{\OC_\lambda(b)}. && \qed
\end{flalign*}
\end{proposition}

Given a full subcategory $\calA \subset \Fuk(X)_\lambda$ we can compose $\CO_\lambda$ with the restriction map
\[
\HH^*(\Fuk(X)_\lambda) \to \HH^*(\calA)
\]
to give a homomorphism $\CO_\calA : \QH^*(X) \to \HH^*(\calA)$.  Similarly we can compose $\OC_\lambda$ with the homomorphism $\HH_*(\calA) \to \HH_*(\Fuk(X)_\lambda)$ induced by the inclusion functor to give a homomorphism $\OC_\calA : \HH_*(\calA) \to \QH^{*+n}(X)$.


\subsection{The weak proper Calabi--Yau structure}
\label{sscwpCY}

Sheridan shows \cite[Section 2.8]{SheridanFano} that $\Fuk(X)_\lambda$ also carries an \emph{$n$-dimensional weak proper Calabi--Yau (wpCY) structure} $[\phi] \in \HH_n(\Fuk(X)_\lambda)^\vee$.

\begin{lemma}[{\cite[Lemma A.2]{SheridanFano}}]
\label{lemwpCY}
This induces a perfect pairing
\[
\langle{-}, {-}\rangle_\mathrm{wpCY} : \HH^*(\Fuk(X)_\lambda) \times \HH_{n-*}(\Fuk(X)_\lambda) \to \kk \quad \text{via} \quad \langle \psi, b \rangle_\mathrm{wpCY} = [\phi] (\psi \cap b).
\]
The same arguments show that the pairing remains perfect after restricting to a full subcategory $\calA$ on a subset of Lagrangians.  We denote this restricted pairing by $\langle{-}, {-}\rangle_{\mathrm{wpCY}, \calA}$.\hfill$\qed$
\end{lemma}

He then shows that $\CO_\lambda$ and $\OC_\lambda$ are adjoint with respect to this pairing, in the following sense.

\begin{proposition}[{\cite[Proposition 2.6]{SheridanFano}}]
\label{propCOOCduality}
For $\alpha \in \QH^*(X)$ and $b \in \HH_{n-*}(\Fuk(X)_\lambda)$ we have
\[
\langle \CO_\lambda(\alpha), b\rangle_\mathrm{wpCY} = \langle \alpha, \OC_\lambda(b) \rangle_X.
\]
By the same arguments, for any full subcategory $\calA \subset \Fuk(X)_\lambda$ on a subset of Lagrangians, the maps $\CO_\calA$ and $\OC_\calA$ are likewise adjoint with respect to $\langle {-}, {-}\rangle_{\mathrm{wpCY}, \calA}$.\hfill$\qed$
\end{proposition}


\subsection{Decompositions and the generation}

For each object $T$ in $\Fuk(X)_\lambda$ let
\[
\CO_T : \QH^*(X) \to \HH^*(\hom^*(T, T))
\]
denote the restriction of $\CO_\lambda$ to $T$; this is just $\CO_\calA$ in the case where $\calA$ contains a single object, $T$.  Let
\[
\CO^0_T : \QH^*(X) \to \rH^*(\hom^*(T, T))
\]
denote the projection of $\CO_T$ to length-zero Hochschild cochains.  Given an idempotent $e \in \QH^*(X)$ of even degree, or equivalently a factor $Q = e \qprod \QH^*(X)$ of $\QH^*(X)$, there is an associated full subcategory $\Fuk(X)_{\lambda,Q}$ of $\Fuk(X)_\lambda$, comprising those objects $T$ for which $\CO^0_T(e) = 1_T \in \rH^*(\hom^*(T, T))$.

\begin{theorem}
\label{thmGen}
Suppose $\LL$ is a Lagrangian in $\Fuk(X)_\lambda$, and that its endomorphism $A_\infty$-algebra $\calA$ is smooth.  The following then hold:
\begin{enumerate}
\item\label{genOC} $\OC_\calA$ gives an isomorphism from $\HH_*(\calA)$ onto its image $Q \subset \QH^*(X)$.
\item\label{genEspace} If $\Char \kk \neq 2$ or $L$ is orientable then $Q$ is contained in the generalised $\lambda$-eigenspace, $\QH^*(X)_\lambda$, of $c_1\qprod$.
\item\label{genVSdecomp} As a $\ZZ/2$-graded vector space, $\QH^*(X)$ is the internal direct sum $Q \oplus Q^\perp$.
\item\label{genAlgDecomp} Both $Q$ and $Q^\perp$ are two-sided ideals in $\QH^*(X)$, so we obtain a decomposition of $\ZZ/2$-graded algebras $\QH^*(X) = Q \times Q^\perp$.
\item\label{genIdemp} The $Q$-component $e$ of $1_X \in \QH^*(X)$ is an even-degree idempotent, satisfying $Q = e \qprod \QH^*(X)$ and $Q^\perp = (1_X-e) \qprod \QH^*(X)$.
\item\label{genCO} $\CO_{\LL} = \CO_\calA$ induces an algebra isomorphism $Q \to \HH^*(\calA)$, and $\ker \CO_{\LL} = Q^\perp$.
\item\label{genSplitGen} $\LL$ lies in $\Fuk(X)_{\lambda, Q}$ and split-generates it.
\end{enumerate}
\end{theorem}

\begin{proof}
\ref{genOC} The algebra $\calA$ is smooth by assumption, and proper because we're dealing with compact Lagrangians, so by \cref{propSmoothProper} $\HH_*(\calA)$ is finite-dimensional and its Shklyarov pairing is perfect.  Since $\OC_\lambda$ is an isometry with respect to this pairing, in the sense of \cref{propOCisometry}, we deduce that it is injective on $\HH_*(\calA)$ and hence induces an isomorphism onto $Q$.

\ref{genEspace} This follows immediately from \cref{propEspace}.

\ref{genVSdecomp} From the proof of \ref{genOC}, the restriction of the Poincar\'e pairing to $Q$ is perfect.  The vector space decomposition of $\QH^*(X)$ then follows from the elementary linear algebra result that if $V$ is a finite-dimensional vector space, $\beta$ is an arbitrary bilinear form on $V$, and $W \subset V$ is a subspace on which $\beta$ is non-degenerate, then $V = W \oplus W^{\perp_\beta}$.  The fact that $Q$ is actually a $\ZZ/2$-graded subspace is an immediate consequence of $\OC_\lambda$ being a $\ZZ/2$-graded map.  $Q^\perp$ is then automatically $\ZZ/2$-graded since the Poincar\'e pairing is $\ZZ/2$-graded.  Note also that since the Poincar\'e pairing is symmetric on even-degree elements and skew-symmetric on odd-degree elements we have $Q^\perp = {}^\perp Q$.

\ref{genAlgDecomp} Suppose we can show that $Q$ is a left ideal.  The Frobenius algebra property of $\QH^*(X)$ then tells us that ${}^\perp Q$ ($=Q^\perp$) is a right ideal.  Since both $Q$ and $Q^\perp$ are $\ZZ/2$-graded subspaces, graded-commutativity of $\QH^*(X)$ then implies that both are in fact two-sided ideals, and we're done.

It therefore suffices to show that $Q$ is a left ideal, so take $q \in Q$ and $\alpha \in \QH^*(X)$.  We wish to show that $\alpha q \in Q$.  By definition of $Q$ there exists $a \in \HH_*(\calA)$ with $q = \OC_\calA(a)$.  We can write this as $\OC_\lambda(i_*a)$, where $i_* : \HH_*(\calA) \to \HH_*(\Fuk(X)_\lambda)$ is the homomorphism induced by inclusion $\calA \to \Fuk(X)_\lambda$.  Using the fact that $\OC_\lambda$ is a $\QH^*(X)$-module map we obtain
\[
\alpha q = \alpha \OC_\lambda(i_*a) = \OC_\lambda(\CO_\lambda(\alpha) \cap i_*a).
\]
The image of $i_*$ is manifestly a $\HH^*(\Fuk(X)_\lambda)$-submodule of $\HH_*(\Fuk(X)_\lambda)$, so it contains $\CO_\lambda(\alpha) \cap i_*a$.  This means that $\alpha q$ is in the image of $\OC_\lambda \circ i_* = \OC_\calA$, which is precisely $Q$, as wanted.

\ref{genIdemp} It follows from \ref{genVSdecomp} that $1_X$ splits into even-degree pieces $e \in Q$ and $1_X-e \in Q^\perp$.  The algebra decomposition in \ref{genAlgDecomp} then tells us that $e (1_X - e) = 0$, so $e$ is idempotent.  It also tells us that $e$ and $1_X - e$ are the units in $Q$ and $Q^\perp$ respectively, so $Q = e \QH^*(X)$ and $Q^\perp = (1-e) \QH^*(X)$.

\ref{genCO} Let $\eps \in \HH_*(\calA)$ be the unique element satisfying $e = \OC_\calA(\eps)$.  Again writing $\OC_\calA$ as $\OC_\lambda \circ i_*$ and using the fact that $\OC_\lambda$ is a $\QH^*(X)$-module map, we see that for all $\alpha \in Q$
\[
\alpha = \alpha e = \OC_\lambda(\CO_\lambda(\alpha) \cap i_*\eps) = \OC_\lambda(i_*(\CO_\calA(\alpha) \cap \eps)).
\]
This forces $\CO_{\LL} = \CO_\calA$ to be injective on $Q$.  By now combining \ref{genOC} with the equality $\dim \HH^*(\calA) = \dim \HH_{n-*}(\calA)$ resulting from \cref{lemwpCY} we obtain $\dim \HH^*(\calA) = \dim Q$, and deduce that $\CO_{\LL}$ is actually an isomorphism $Q \to \HH^*(\calA)$.

It remains to show that $\ker \CO_{\LL} = Q^\perp$, and since we know
\[
\dim Q^\perp = \dim \QH^*(X) - \dim Q \leq \dim \QH^*(X) - \operatorname{rank} \CO_{\LL} = \dim \ker \CO_{\LL}
\]
it's enough to show that $\ker \CO_{\LL} \subset Q^\perp$.  Suppose then that $\alpha \in \ker \CO_{\LL}$.  Since $\CO_{\LL}$ is an algebra homomorphism we get $\CO_{\LL}(e \alpha) = 0$.  But $e \alpha$ is in $Q$, so by injectivity of $\CO_{\LL}$ on $Q$ we deduce that $e \alpha = 0$, and hence $\alpha = (1_X - e) \alpha$.  Thus $\alpha \in (1_X - e) \QH^*(X) = Q^\perp$, and so $\ker \CO_{\LL} \subset Q^\perp$.

\ref{genSplitGen} To prove that $\LL$ lies in $\Fuk(X)_{\lambda,Q}$ we need to show that $\CO^0_{\LL}(e) = 1_{\LL}$, for which it suffices to show that $\CO_{\LL}(e)$ is the unit $1_{\HH^*(\calA)}$ in $\HH^*(\calA)$.  To do this, note that since $\CO_{\LL}$ induces an isomorphism $Q \to \HH^*(\calA)$ there exists a unique $e' \in Q$ such that $\CO_{\LL}(e') = 1_{\HH^*(\calA)}$.  We then have
\[
\CO_{\LL}(e') = \CO_{\LL}(e e') = \CO_{\LL}(e) \cup \CO_{\LL}(e') = \CO_{\LL}(e).
\]
Injectivity of $\CO_{\LL}$ on $Q$ gives $e' = e$, and hence $\CO_{\LL}(e) = 1_{\HH^*(\calA)}$, as wanted.

For any other object $T$ in $\Fuk(X)_{\lambda, Q}$, we have $\CO^0_T(e) = 1_T$ by definition.  So $\CO^0_T \circ \OC_\calA(\eps) = 1_T$, where $\eps \in \HH_*(\calA)$ is the unique element with $\OC_\calA(\eps) = e$, as above.  The fact that $T$ is split-generated by $\LL$ then follows from Abouzaid's generation criterion \cite[Proposition 1.3, Lemma 1.4]{AbouzaidGeometricCriterion}.
\end{proof}

\begin{remark}
An alternative approach is to modify the argument of \cite[Theorem 4.10]{Sanda} as follows (\cite{Sanda} assumes the existence of a cyclic structure on $\Fuk(X)_\lambda$, which we circumvent using the wpCY structure).  By the first part of the above proof, $\OC_\calA$ is an isomorphism from $\HH_*(\calA)$ onto its image, on which $\langle {-}, {-} \rangle_X$ is perfect.  For $a$ and $b$ in $\HH_*(\calA)$ \cref{propCOOCduality} gives
\[
\langle \CO_\calA \circ \OC_\calA(a), b \rangle_{\mathrm{wpCY}, \calA} = \langle \OC_\calA(a), \OC_\calA(b) \rangle_X,
\]
and $\langle {-}, {-} \rangle_{\mathrm{wpCY}, \calA}$ is perfect by \cref{lemwpCY}.  We conclude that $\CO_\calA \circ \OC_\calA$ induces an isomorphism
\[
\HH_{*-n}(\calA) \to \HH^*(\calA).
\]
Let $\eps \in \HH_{-n}$ be the unique element with $\CO_\calA \circ \OC_\calA (\eps) = 1_{\HH^*(\calA)}$, and let $e \in \QH^0(X)$ be $\OC_\calA(\eps)$.  We then have
\[
e^2 = e \OC_\calA(\eps) = \OC_\lambda(i_*(\CO_\calA(e) \cap \eps)) = \OC_\lambda(i_*(1_{\HH^*(\calA)} \cap \eps)) = \OC_\calA(\eps) = e,
\]
so $e$ is an idempotent and hence induces an algebra decomposition $\QH^*(X) = Q \times Q^\perp$, where $Q = e \QH^*(X)$ and $Q^\perp = (1_X - e) \QH^*(X)$.  One can then complete the proof similarly to above.
\end{remark}


\section{Proof of the main results}
\label{secProofMain}


\subsection{A single isolated critical point}
\label{sscLMF}

Suppose $L \subset X$ is a monotone Lagrangian torus with superpotential $W_L$, viewed as an element of $S = \kk[\rH_1(L; \ZZ)]$.  Suppose $\calL$ is an isolated critical point of $W_L$, in the sense of \cref{lemIsolatedTFAE}, and let $\lambda = W_L(\calL)$ be the corresponding critical value.  Then let $\LL \in \Fuk(X)_\lambda$ be the object obtained by equipping $L$ with the local system $\calL$.  We denote the $A_\infty$-algebra $\hom^*(\LL, \LL) = \CF^*(\LL, \LL)$ by $\calA$, which we also view as an $A_\infty$-category with one object.

In \cite{ChoHongLauTorus} Cho--Hong--Lau construct a \emph{localised mirror functor}
\[
\LMF : \Fuk(X)_\lambda \to \mf(S, W_L - \lambda).
\]
The critical point $\calL$ of $W_L$ corresponds to a point in $\Spec S$, whose skyscraper sheaf defines an object $\calE$ in $\mf(S, W_L - \lambda)$.  Letting $\calB$ denote the endomorphism dg-algebra of $\calE$, it is shown in \cite[Theorem 5]{SmithSuperfiltered} that $\LMF$ induces a quasi-isomorphism $\Phi : \calA \to \calB$.

\begin{remark}
A priori the functor is defined on actual monotone Lagrangians in $\Fuk(X)_\lambda$, rather than on more general summands of twisted complexes, but this is enough for our purposes.
\end{remark}

Let $S_\calL$ denote the localisation of $S$ at the maximal ideal $\m$ corresponding to $\calL$, and let $\hatS$ denote the $\m$-adic completion of $S$.  We can consider the analogues of $\calE$ and $\calB$ over $S_\calL$ and $\hatS$, which we denote by adding a subscript $\calL$ or a hat.

\begin{lemma}
\label{lemBqis}
The natural maps $\calB \to \calB_\calL \to \hatB$ are quasi-isomorphisms.
\end{lemma}

\begin{proof}
The algebra $\calB$ is a finite rank free module over $S$, and $\calB_\calL$ and $\hatB$ are obtained by tensoring over $S$ with $S_\calL$ and $\hatS$ respectively.  By \cref{lemCommAlg} (using the fact that $S_\calL$ is Noetherian), the induced maps $\rH^*(\calB) \to \rH^*(\calB_\calL) \to \rH^*(\hatB)$ are
\[
\rH^*(\calB) \to \rH^*(\calB) \otimes_S S_\calL \to \rH^*(\calB) \otimes_S \hatS.
\]
These are isomorphisms since $\rH^*(\calB)$ is annihilated by $\m$ (because $\calE$ is).
\end{proof}

Because $\calL$ is an isolated critical point of $W_L$, the category $\mf(S_\calL, W_L - \lambda)$ falls into the setting considered by Dyckerhoff in \cite{Dyckerhoff} (more precisely, $(S_\calL, W_L - \lambda)$ satisfies his condition (B)), where he shows that $\calE_\calL$ split-generates $\mf(S_\calL, W_L - \lambda)$ \cite[Corollary 5.3]{Dyckerhoff}.  Note that our $S_\calL$ and $\calE_\calL$ correspond to Dyckerhoff's $R$ and $k^\mathrm{stab}$ respectively, that our $S$ is different from his, and that he writes $\widehat{\calC}_\mathrm{pe}$ for what we denote by $\PiTw \calC$.  Combining the above ideas allows us to prove the following.

\begin{proposition}
\label{propCategoryDescriptions}
We have quasi-equivalences
\begin{equation}
\label{eqqeq}
\PiTw \calA \simeq \operatorname{\Pi} \mf(S_\calL, W_L - \lambda) \simeq \mf(\hatS, W_L - \lambda).
\end{equation}
In other words, the full subcategory of $\Fuk(X)_\lambda$ split-generated by $\LL$ is quasi-equivalent to $\mf(\hatS, W_L - \lambda)$.  Under this quasi-equivalence, $\LL$ is sent to the matrix factorisation $\hatE$ associated to the skyscraper sheaf at $\calL$.
\end{proposition}
\begin{proof}
From the quasi-isomorphisms $\Phi : \calA \to \calB$ and $\calB \to \calB_\calL$ we get that $\PiTw \calA \simeq \PiTw \calB_\calL$.  Dyckerhoff's split-generation result tells us that $\PiTw \calB_\calL \simeq \Pi \mf(S_\calL, W_L - \lambda)$ (there is no need for a $\Tw$ here since matrix factorisation categories are already pretriangulated), so the first quasi-equivalence in \eqref{eqqeq} follows.  The second is then a consequence of Dyckerhoff's description \cite[Theorem 5.7]{Dyckerhoff} of idempotent completion as completion of the ring.
\end{proof}

We also obtain the following two results.

\begin{proposition}
\label{propAsmooth}
The $A_\infty$-algebra $\calA$ is smooth.
\end{proposition}
\begin{proof}
Dyckerhoff shows that $\mf(S_\calL, W_L - \lambda)$ is smooth \cite[Section 7]{Dyckerhoff}.  The result then follows from \cref{propCategoryDescriptions}, by Morita-invariance of smoothness.
\end{proof}

\begin{proposition}
\label{propAHH}
$\HH^*(\calA)$ is isomorphic to the localisation $\Jac_\calL W_L$ of the Jacobian algebra of $W_L$ at $\calL$, concentrated in even degree.
\end{proposition}
\begin{proof}
By \eqref{eqqeq} and Morita invariance we have $\HH^*(\calA) \cong \HH^*( \mf (S_\calL, W_L - \lambda) )$.  The latter is computed by Dyckerhoff \cite[Corollary 6.5]{Dyckerhoff} to be exactly $\Jac_\calL W_L$.

Alternatively, we have $\HH^*(\calA) \cong \HH^*(\calB)$ and the latter is computed in \cite[Theorem 4]{SmithSuperfiltered} to be the Jacobian of $W_L$ over $\hatS$.  By an analogous argument to that at the end of the proof of \cref{propphii}, this is equivalent to the localisation $\Jac_\calL W_L$.
\end{proof}

Plugging \cref{propCategoryDescriptions,propAsmooth} into \cref{thmGen}, we obtain the following.

\begin{theorem}
\label{thmSingle}
$\QH^*(X)$ decomposes as a product of $\ZZ/2$-graded algebras
\[
Q \times Q^\perp = e \QH^*(X) \times (1_X-e) \QH^*(X),
\]
where $\CO_{\LL}$ has kernel $Q^\perp$ and induces an isomorphism $Q \to \HH^*(\calA) \cong \Jac_\calL W_L$.  The summand $\Fuk(X)_{\lambda, Q}$ of the Fukaya category contains $\LL$, which split-generates it, and is quasi-equivalent to $\mf (\hatS, W_L - \lambda)$.  Under this quasi-equivalence, $\LL$ corresponds to the matrix factorisation $\hatE$ associated to the skyscraper sheaf at $\calL$.\hfill$\qed$
\end{theorem}


\subsection{The structure of toroidal subcategories}
\label{sscToroidal}

Before moving on from the case of a single critical point, we make a slight detour to prove two structural results about toroidal subcategories of the Fukaya category.  Recall from \cref{defToroidal} that these are summands split-generated by a monotone torus equipped with an isolated critical point of its superpotential.  The reader only interested in quantum cohomology may skip to \cref{sscCombining}.

\begin{proposition}
\label{propSegal}
Given a toroidal subcategory $\calC$, any non-zero object in $\calC$ split-generates it.
\end{proposition}

\begin{proof}
By \cref{thmSingle} the category $\calC$ is quasi-equivalent to a category of the form $\mf(\hatS, W_L - \lambda)$, continuing the notation of the previous subsection.  We may thus work in the latter category instead.

So let $\calG$ be an arbitrary non-zero object in $\mf(\hatS, W_L - \lambda)$, and let
\begin{gather*}
i^* : \mf(\hatS, W_L - \lambda) \to \mf(\hatS / \m, 0)
\\ i_* : \mf(\hatS/\m, 0) \to \mf(\hatS, W_L - \lambda)
\end{gather*}
be the pullback (restriction) and pushforward (inclusion) functors associated to the inclusion of the closed point $\calL$ in $\Spec \hatS$.  Consider then the object $i_*i^* \calG$.  On the one hand, $i^*\calG$ can only be a direct sum of shifts of $\hatS/\m$, so $i_*i^*\calG$ is a direct sum of shifts of $\hatE$.  We deduce that $i_*i^*\calG$ split-generates $\hatE$.  On the other hand, $i_*i^*\calG$ is the (derived) tensor product of $\calG$ with the Koszul resolution of the skyscraper sheaf at $\calL$.  This in turn is a twisted complex built from $\calG$, so we deduce that $\calG$ split-generates $i_*i^*\calG$.

Combining these two deductions we see that $\calG$ split-generates $\hatE$.  We know already from \cite[Corollary 5.3]{Dyckerhoff} that $\hatE$ split-generates $\mf(\hatS, W_L - \lambda)$, so we conclude that $\calG$ also split-generates the category, which is what we want.
\end{proof}

I am grateful to Ed Segal for suggesting this argument.

Recall, again from \cref{defToroidal}, that a toroidal generator for the toroidal category $\calC$ is any object of the form $\LL$ (i.e.~a monotone torus equipped with an isolated critical point of its superpotential) that split-generates $\calC$.  Equivalently, by \cref{propSegal}, it is any \emph{non-zero} object in $\calC$ of the form $\LL$.  We conjecture that any two toroidal generators are quasi-isomorphic, up to a shift, and can prove the following partial result.

\begin{proposition}
\label{propToriIsomc}
If $T_1$ and $T_2$ are toroidal generators of a toroidal subcategory $\calC$ then we have a quasi-isomorphism
\[
T_1^{\oplus 2^{n-1}} \oplus T_1[1]^{\oplus 2^{n-1}} \simeq T_2^{\oplus 2^{n-1}} \oplus T_2[1]^{\oplus 2^{n-1}}.
\]
\end{proposition}

\begin{proof}
Following the notation of \cref{propSegal}, we may assume that $T_1$ is $\LL$ and hence corresponds to $\hatE$ in $\mf(\hatS, W_L - \lambda)$.  We abbreviate $T_2$ to $T$ and write $\calT$ for the corresponding matrix factorisation.

We claim that for every element $x$ of the ideal $\m$, the endomorphism of $\calT$ given by scalar multiplication by $x$ is nullhomotopic (i.e.~exact).  Assuming this for now, consider the object $i_*i^*\calT$, mimicking the proof of \cref{propSegal}.

First view $i_*i^*\calT$ as the tensor product of $\calT$ with the Koszul resolution of the skyscraper sheaf at $\calL$.  This Koszul resolution looks like $2^{n-1}$ copies of $\hatS$ and $2^{n-1}$ copies of $\hatS[1]$, connected by maps which are each given by scalar multiplication by an element of $\m$.  By our claim we can thus represent $i_*i^*\calT$ by $2^{n-1}$ copies of $\calT$ and $2^{n-1}$ copies of $\calT[1]$, connected by nullhomotopic maps.  We deduce that
\begin{equation}
\label{eqTT}
i_*i^*\calT \simeq \calT^{\oplus 2^{n-1}} \oplus \calT[1]^{\oplus 2^{n-1}}.
\end{equation}

Now instead use the fact that $i^*\calT$ is a sum of copies of $\hatS/\m$ and $\hatS/\m[1]$.  More precisely, suppose
\[
\calT = \cdots \xto{f^1} \hatS^{\oplus k} \xto{f^0} \hatS^{\oplus k} \xto{f^1} \hatS^{\oplus k} \xto{f^0} \cdots,
\]
where the $f^i$ are $\hatS$-linear maps satisfying $f^{i+1} \circ f^{i} = W_L \cdot \id_{\hatS^k}$.  We then get
\begin{equation}
\label{eqTS}
i^*\calT = \cdots \xto{\overline{f}^1} (\hatS/\m)^{\oplus k} \xto{\overline{f}^0} (\hatS/\m)^{\oplus k} \xto{\overline{f}^1} (\hatS/\m)^{\oplus k} \xto{\overline{f}^0} \cdots \simeq (\hatS/\m)^{\oplus l} \oplus (\hatS/\m)[1]^{\oplus l},
\end{equation}
where $\overline{f}^i$ denotes the reduction of $f^i$ modulo $\m$ and $l$ denotes
\[
\dim_\kk (\ker f^1 / \operatorname{im} f^0) = \dim_\kk (\ker f^0 / \operatorname{im} f^1).
\]
Applying $i_*$ to \eqref{eqTS} and using the fact that $i_*(\hatS/\m) \simeq \hatE$ then gives
\begin{equation}
\label{eqTE}
i_*i^* \calT \simeq \hatE^{\oplus k} \oplus \hatE[1]^{\oplus k}.
\end{equation}

The result now follows from \eqref{eqTT} and \eqref{eqTE} if we can show $k = 2^{n-1}$.  To do this, consider the endomorphism algebra of $i_*i^*\calT$.  Computing its dimension using \eqref{eqTT} and \eqref{eqTE} gives
\[
2^{2n-2} \dim_\kk \End^*(\hatE) = k^2 \dim_\kk \End^*(\calT).
\]
Since $\End^*(\calT) \cong \HF^*(T, T)$ and $\End^*(\hatE) \cong \HF^*(\LL, \LL)$, both of which are $2^n$-dimensional, we deduce that $k = 2^{n-1}$, as needed.

It remains to prove the claim, namely that each element of $\m$ acts nullhomotopically on $\calT$.  To do this, recall from \cref{propAHH} and its proof that $\HH^*(\calC)$ is the localisation (or, equivalently, completion) $\Jac_\calL W_L$ of $\Jac W_L$ at $\calL$.  Moreover, we know $\HH^*(\calC)$ acts on each object $O$ of $\calC$ by projection to length zero $\pi : \HH^*(\calC) \to \End^*(O)$, and it follows from \cite[p266]{Dyckerhoff} that the corresponding action of $\Jac_\calL W_L$ on $\calC$ is by scalar multiplication.  Explicitly, the following diagram commutes
\[
\begin{tikzcd}[column sep=5em, row sep=2.5em]
\hatS \arrow[r, "\text{scalar mult}"] \arrow[d, "J"] & \End^*(O) \\
\Jac_\calL W_L \arrow[r, phantom, "\cong"] & \HH^*(\calC), \arrow[u, swap, "\pi"]
\end{tikzcd}
\]
where $J$ is reduction modulo the Jacobian ideal (so in particular all elements of this ideal act nullhomotopically on every $O$).  We can therefore rephrase our claim as follows: when $O = \calT$, the unique maximal ideal in $\HH^*(\calC)$ is in the kernel of $\pi$.  Since $T_1$ and $T_2$, or equivalently $\hatE$ and $\calT$, were both equally defined to be toroidal generators of $\calC$, by symmetry it suffices to prove this rephrased claim for $O = \hatE$.  And in this case the claim can be proved directly, either by viewing $\hatE$ as the skyscraper sheaf at $\calL$ in the singularity category of $W_L - \lambda$, or by the computation in \cite[Section 5]{SmithSuperfiltered}, where reduction modulo $\m$ is identified with projection to length zero.
\end{proof}


\subsection{Combining isolated critical points}
\label{sscCombining}

Returning to the main thread, our next task is to show that the splittings of $\QH^*(X)$ arising from different critical points of $W_L$ are compatible with each other, and that the corresponding objects do not interact.

Suppose then that the isolated critical points of $W_L$ are $\calL_1, \calL_2, \dots, \calL_r$.  Let $\lambda_i = W_L(\calL_i)$, and let $\LL_i$ be the object $(L, \calL_i)$ in $\Fuk(X)_{\lambda_i}$.  By \cref{thmSingle}, for each $i$ we get a splitting
\[
\QH^*(X) = Q_i \times Q_i^\perp = e_i \QH^*(X) \times (1_X-e_i) \QH^*(X).
\]

\begin{proposition}
\label{propIdempotentsOrthogonal}
The idempotents $e_i$ are pairwise algebraically orthogonal, i.e.~$e_a e_b = 0$ for all $a$ and $b$ with $a \neq b$.  (This implies, by the Frobenius algebra property, that the $Q_i$ are pairwise geometrically orthogonal with respect to the Poincar\'e pairing.)
\end{proposition}
\begin{proof}
Fix distinct $a$ and $b$.  First suppose $\lambda_a \neq \lambda_b$.  By \cref{thmGen}\ref{genEspace} there exist positive integers $m_a$ and $m_b$ such that $(c_1 - \lambda_a)^{m_a} e_a = 0$ and $(c_1 - \lambda_b)^{m_b} e_b = 0$.  Since the polynomials $(T-\lambda_a)^{m_a}$ and $(T-\lambda_b)^{m_b}$ are coprime there exist polynomials $f(T)$ and $g(T)$ such that
\[
1 = f(T)(T-\lambda_a)^{m_a} + g(T)(T-\lambda_b)^{m_b}.
\]
Plugging in $T = c_1$, and multiplying both sides by $e_a$ and $e_b$ (which have even degree so commute with everything), we get
\[
e_a e_b = f(c_1) \big((c_1 -\lambda_a)^{m_a} e_a\big) e_b + g(c_1) \big((c_1 - \lambda_b)^{m_b} e_b\big) e_a = 0,
\]
as wanted.

It remains to consider the case where $\lambda_a$ and $\lambda_b$ are equal to some common value $\lambda$.  In this case $\LL_a$ and $\LL_b$ both lie in $\Fuk(X)_\lambda$, and it is well known that they are orthogonal in the sense that $\HF^*(\LL_a, \LL_b) = 0$.  This can be proved by considering the Oh spectral sequence \cite{OhSS}
\[
E_1 = \rH^*(L; \calL_a^{-1} \otimes_\kk \calL_b) \implies \HF^*(\LL_a, \LL_b),
\]
and noting that the cohomology group $\rH^*(L; \calL_a^{-1} \otimes_\kk \calL_b)$ with local coefficients $\calL_a^{-1} \otimes_\kk \calL_b$ vanishes.  The full subcategories $\calA_a$ and $\calA_b$ on $\LL_a$ and $\LL_b$ are therefore categorically orthogonal, so by the discussion in \cref{sscShklyarov} the images of $\HH_*(\calA_a)$ and $\HH_*(\calA_b)$ in $\HH_*(\Fuk(X)_\lambda)$ are orthogonal with respect to the Shklyarov pairing.  By \cref{propOCisometry} we then deduce that the $Q_a$ and $Q_b$ are orthogonal with respect to the Poincar\'e pairing.

Now, $e_a e_b$ lies in $e_a \QH^*(X) = Q_a$, so it must be orthogonal to every element of $Q_b$.  But it also lies in $e_b \QH^*(X) = Q_b$, on which the Poincar\'e pairing is non-degenerate.  We conclude that $e_a e_b = 0$.
\end{proof}

Combining this with \cref{thmSingle}, and letting $\calA_i = \CF^*(\LL_i, \LL_i)$, we get the following.

\begin{corollary}
\label{thmSeveral}
$\QH^*(X)$ decomposes as a product of $\ZZ/2$-graded algebras
\[
Q_1 \times \dots \times Q_r \times \bigcap_{i=1}^r Q_i^\perp = e_1 \QH^*(X) \times \dots \times e_r \QH^*(X) \times \Big(\prod_{i=1}^r (1_X-e_i)\Big) \QH^*(X).
\]
Each map
\[
\CO_{\LL_i} : \QH^*(X) \to \HH^*(\calA_i) \cong \Jac_{\calL_i} W_L
\]
induces an isomorphism from $Q_i = e_i \QH^*(X)$ and annihilates the other factors.  Consequently we have
\[
Q_1 \times \dots \times Q_r \cong \Jac_{\calL_1} W_L \times \dots \times \Jac_{\calL_r} W_L \cong \Jacisol W_L.
\]
Each summand $\Fuk(X)_{\lambda_i, Q_i}$ of the Fukaya category is described in \cref{thmSingle}; in particular, it contains and is split-generated by $\LL_i$.\hfill$\qed$
\end{corollary}


\subsection{Constructing the map geometrically}
\label{sscCOL}

\Cref{thmSeveral} tells us that the map
\begin{equation}
\label{eqProductCO}
\CO_{\LL_1} \times \dots \times \CO_{\LL_r} : \QH^*(X) \to \HH^*(\calA_1) \times \dots \times \HH^*(\calA_r)
\end{equation}
is surjective and induces an algebra decomposition of $\QH^*(X)$ into the kernel times the image.  We also know that the image is isomorphic to $\Jacisol W_L$, although this isomorphism is somewhat indirect.  Our next aim is to give a more geometric description of this map \eqref{eqProductCO}, which goes directly to $\Jacisol W_L$ and avoids any categorical $A_\infty$-algebraic constructions, and hence give a more concrete realisation of this algebra decomposition.

In \cite{SmithHHviaHF} we consider the Floer cohomology $\HF^*_S(\bL, \bL)$ of $L$ with coefficients in $S = \kk[\rH_1(L; \ZZ)]$.  Explicitly, this means that we work over $S$ and weight the contribution of each pseudoholomorphic disc $u$ by the monomial in $S$ corresponding to its boundary homology class, as discussed in \cref{sscStatement}.  There is similarly a unital $\kk$-algebra homomorphism
\[
\CO^0_{\bL} : \QH^*(X) \to \HF^*_S(\bL, \bL),
\]
also described in \cref{sscStatement}.

\begin{remark}
\label{opposite}
The map $\CO^0_\bL$ always lands in the graded-centre of $\HF^*_S(\bL, \bL)$, so one could equally take its codomain to be the graded-opposite algebra $\HF^*_S(\bL, \bL)^\op$.  This will be useful later.  We shall really only need this map after localising at isolated critical points, which will make the codomain graded-commutative anyway.
\end{remark}

Recall from \cref{defSisol} that $S_\isol$ is the localisation of $S$ at the set $\{\calL_1, \dots, \calL_r\}$ of isolated critical points of $W_L$.  Concretely this is just the ring obtained from $S$ by inverting every $f \in S$ that satisfies $f(\calL_i) \neq 0$ for all $i$, and can be viewed as the product $S_{\calL_1} \times \dots \times S_{\calL_r}$ of the localisations at the individual $\calL_i$.

\begin{definition}
\label{defHFisol}
Let $\HF^*_\isol(\bL, \bL)$ and
\[
\COisol : \QH^*(X) \to \HF^*_\isol(\bL, \bL)
\]
denote the analogues of $\HF^*_S(\bL, \bL)$ and $\CO^0_\bL$ constructed with $S_\isol$ in place of $S$.  By \cref{lemCommAlg}\ref{itmLocalisation} we could equivalently define them by starting with $\HF^*_S(\bL, \bL)$ and $\CO^0_\bL$ and either localising at the isolated critical points or tensoring with $S_\isol$.  Similarly, for each $i$ let $\HF^*_{\calL_i}(\bL, \bL)$ and $\CO^0_{\bL, \calL_i}$ denote the analogues of $\HF^*_S(\bL, \bL)$ and $\CO^0_\bL$ constructed with $S_{\calL_i}$ in place of $S$, or equivalently their localisations at $\calL_i$ or tensor products with $S_{\calL_i}$.
\end{definition}

The technical heart of this subsection is the following pair of results, which we shall prove shortly.

\begin{lemma}
\label{lemJacisol}
As an $S_\isol$-algebra, $\HF^*_\isol(\bL, \bL)$ is generated by the unit and is isomorphic to $\Jacisol W_L$.
\end{lemma}

\begin{proposition}
\label{propphii}
For each $i$ there is an algebra isomorphism
\[
\phi_i : \HH^*(\calA_i) \to \HF^*_{\calL_i}(\bL, \bL),
\]
such that $\phi_i \circ \CO_{\LL_i} = \CO^0_{\bL, \calL_i}$.
\end{proposition}

From these we can deduce the result we really want.

\begin{corollary}
\label{corThmA}
$\COisol : \QH^*(X) \to \HF^*_\isol(\bL, \bL)$ is surjective and induces an algebra decomposition of $\QH^*(X)$ into the kernel times the image.  Moreover, the image is equal to $\Jacisol W_L$ times the unit.
\end{corollary}

\begin{proof}
Apply the maps $\phi_1, \dots, \phi_r$ from \cref{propphii} to the $r$ factors in \eqref{eqProductCO}, to see that
\begin{multline}
\label{eqProductCOLi}
\CO^0_{\bL, \calL_1} \times \dots \times \CO^0_{\bL, \calL_r} : \QH^*(X) \to \HF^*_{\calL_1}(\bL, \bL) \times \dots \times \HF^*_{\calL_r}(\bL, \bL)
\\ = \HF^*_S(\bL, \bL) \otimes_S (S_{\calL_1} \times \dots \times S_{\calL_r})
\end{multline}
is surjective and induces an algebra decomposition of $\QH^*(X)$.  From the description
\[
S_\isol = S_{\calL_1} \times \dots \times S_{\calL_r}
\]
we obtain an identification of the map \eqref{eqProductCOLi} with
\[
\CO^0_{\bL,\isol} : \QH^*(X) \to \HF^*_\isol(\bL, \bL) = \HF^*_S(\bL, \bL) \otimes_S S_\isol,
\]
and by \cref{lemJacisol} the image is $\Jacisol W_L$ times the unit.
\end{proof}

From the definition of $\CO^0_\bL$ we see that, as a map to $\Jacisol W_L$, $\CO^0_\isol$ has exactly the geometric description given below \cref{CorollaryB}.

Having seen the utility of \cref{lemJacisol,propphii}, we now give the promised proofs.

\begin{proof}[Proof of \cref{lemJacisol}]
We want to show that the unique $S_\isol$-algebra map $S_\isol \to \HF^*_\isol(\bL, \bL)$ induces an isomorphism $\Jacisol W_L \to \HF^*_\isol(\bL, \bL)$.  Being an isomorphism can be checked locally at maximal ideals, so it suffices to show that for each $i$ the unique $S_{\calL_i}$-algebra map $S_{\calL_i} \to \HF^*_{\calL_i}(\bL, \bL)$ induces an isomorphism $\Jac_{\calL_i} W_L \to \HF^*_{\calL_i}(\bL, \bL)$.

To compute $\HF_{\calL_i}^*(\bL, \bL)$ we will use the Oh spectral sequence \cite{OhSS}
\[
E_1 = \rH^*(L; S_{\calL_i}) \implies \HF_{\calL_i}^*(\bL, \bL).
\]
This is most easily constructed by using a pearl model for the Floer complex, and filtering the generators (which are critical points of a Morse function) by their Morse index; see \cite{biran2007quantum}.  The $E_1$ differential on the spectral sequence, denoted $\diff_1$, can be described as follows.  Let
\[
W_L = \sum_{\mathclap{\gamma \in \rH_1(L; \ZZ)}} n_\gamma z^\gamma \in S = \kk[\rH_1(L; \ZZ)].
\]
Recall that each $n_\gamma$ is an integer counting the number of index $2$ discs with boundary in class $\gamma$, and the sum is finite by Gromov compactness.  Given a class $b \in \rH^1(L; S_{\calL_i})$, by definition $\diff_1 b$ computes the Morse index $0$ part of the pearl differential of $b$.  And this can be calculated geometrically to be
\[
\sum_{\gamma} \langle \gamma, b \rangle n_\gamma z^{\gamma}
\]
(times the unit), where $\langle \cdot, \cdot \rangle$ is the pairing between $\rH_1(L)$ and $\rH^1(L)$.

Now fix a basis $e_1, \dots, e_n$ for $\rH_1(L; \ZZ)$, with respect to which each $\gamma$ has components $(\gamma_1, \dots, \gamma_n)$, let $z_i$ be the monomial $z^{e_i}$ so that $S = \kk[z_1^{\pm1}, \dots, z_n^{\pm1}]$, and let $b_1, \dots, b_n$ be the dual basis for $\rH^1(L; \ZZ)$.  We then deduce that
\[
\diff_1 (z_i^{-1}b_i) = \sum_\gamma \langle \gamma, z_i^{-1}b_i \rangle n_\gamma z^\gamma = \sum_\gamma \gamma_i n_\gamma z_1^{\gamma_1} \cdots z_i^{\gamma_i - 1} \cdots z_n^{\gamma_n} = \frac{\partial W_L}{\partial z_i}.
\]
Since $E_1 = \rH^*(L; S_{\calL_i})$ can be viewed as the exterior algebra over $S_{\calL_i}$ on generators $z_i^{-1}b_i$, the differential $\diff_1$ is then completely determined by the Leibniz rule.  Explicitly, $(E_1, \diff_1)$ looks like the Koszul complex associated to $(\partial W_L / \partial z_1, \dots, \partial W_L / \partial z_n)$.

To complete the proof it therefore suffices to show that $(\partial W_L / \partial z_1, \dots, \partial W_L / \partial z_n)$ is a regular, hence Koszul-regular, sequence in $S_{\calL_i}$; then the cohomology of the Koszul complex is concentrated in degree zero, where it is manifestly $\Jac_{\calL_i} W_L$.  Since $S_{\calL_i}$ is a regular local ring of Krull dimension $n$, and $\Jac_{\calL_i} W_L$ has Krull dimension $0$ because $\calL_i$ is an isolated critical point, regularity of the sequence $(\partial W_L / \partial z_1, \dots, \partial W_L / \partial z_n)$ follows from \cite[Exercise 26.2.D]{VakilAG}.
\end{proof}

\begin{proof}[Proof of \cref{propphii}]
Throughout the proof we fix an $\calL_i$ and drop all subscript $i$'s, so we're in the setup of \cref{sscLMF}.  Recall from that subsection that the localised mirror functor induces a quasi-isomorphism $\Phi : \calA \to \calB$, where $\calB$ is the endomorphism algebra of the matrix factorisation $\calE$.  This in turn induces an isomorphism $\HH^*(\calA) \to \HH^*(\calA, \calB)$, as in \eqref{eqHHmaps}, which we denote by $\rH(\Phi_*)$.

By restricting \cite[Theorem B]{SmithHHviaHF} to $\LL$ we obtain a commutative diagram combining $\rH(\Phi_*)$ with the maps $\CO_\LL$ and $\CO^0_\bL$:
\begin{equation*}
\begin{tikzcd}[column sep=5em, row sep=2.5em]
\QH^*(X) \arrow{d}{\CO^0_\bL} \arrow{r}{\CO_{\LL}} & \HH^*(\calA) \arrow{d}{\rH(\Phi_*)}
\\ \HF_S^*(\bL, \bL)^\op \arrow{r}{\rH(\Theta_{\LL})} & \HH^*(\calA, \calB).
\end{tikzcd}
\end{equation*}
Here $\Theta_{\LL}$ is a certain cohomologically unital $A_\infty$-algebra homomorphism
\[
\CF^*_S(\bL, \bL)^\op \to \rCC^*(\calA, \calB),
\]
and $\HF_S^*(\bL, \bL)^\op$ is the graded-opposite algebra of $\HF_S^*(\bL, \bL)$; recall \cref{opposite}.

The constructions of $\HF^*_S(\bL, \bL)$, $\CO^0_\bL$, and $\Theta_{\LL}$ still make sense if we work over the completion $\hatS$ of $S$ at the maximal ideal $\m$ corresponding to $\calL$, and we denote the corresponding objects and maps by adding hats.  Recall from \cref{lemBqis} that the natural map $\calB \to \hatB$ is a quasi-isomorphism, so $\rH(\hatPhi_*)$ is still an isomorphism.  The virtue of passing to these completions is that $\rH(\hatTheta)$ is an isomorphism \cite[Theorem C]{SmithHHviaHF}.  We deduce that there is a commutative diagram
\begin{equation*}
\label{eqCompletedDiagram}
\begin{tikzcd}[column sep=5em, row sep=2.5em]
\QH^*(X) \arrow{d}{\hatCO} \arrow{r}{\CO_{\LL}}& \HH^*(\calA) \arrow{d}{\rH(\hatPhi_*)}
\\ \HF_{\hatS}^*(\bL, \bL)^\op \arrow{r}{\rH(\hatTheta)} & \HH^*(\calA, \hatB)
\end{tikzcd}
\end{equation*}
in which the bottom and right-hand arrows are isomorphisms.

The completion map $S_\calL \to \hatS$ on coefficients induces a map
\[
C : \HF^*_\calL(\bL, \bL) \to \HF^*_{\hatS}(\bL, \bL)
\]
($C$ for `completion').  We claim that this is an isomorphism and satisfies $\hatCO = C \circ \CO^0_{\bL, \calL}$.  (We can safely drop all mention of opposite algebras here---see \cref{rmkDropOpposite}.) Then
\[
C^{-1} \circ \rH(\hatTheta)^{-1} \circ \rH(\hatPhi_*)
\]
defines the desired isomorphism $\phi : \HH^*(\calA) \to \HF^*_\calL(\bL, \bL)$.

To prove the claimed properties of $C$, recall from \cref{lemCommAlg}\ref{itmCompletion} that (by exactness of completion) $\HF^*_{\hatS}(\bL, \bL)$ may be viewed as the $\m$-adic completion of $\HF^*_S(\bL, \bL)$, and that this is naturally isomorphic to $\HF^*_S(\bL, \bL) \otimes_S \hatS$.  More precisely, the natural map
\[
\HF^*_S(\bL, \bL) \otimes_S \hatS \to \HF^*_{\hatS}(\bL, \bL)
\]
is an isomorphism, and $\hatCO$ is $\CO^0_\bL$ tensored with $\hatS$.  Similarly, from \cref{defHFisol} (which uses \cref{lemCommAlg}\ref{itmLocalisation}) the natural map
\[
\HF^*_S(\bL, \bL) \otimes_S S_\calL \to \HF^*_\calL(\bL, \bL)
\]
is an isomorphism, and $\CO^0_{\bL, \calL}$ is $\CO^0_\bL$ tensored with $S_\calL$.  We conclude that $C$ corresponds to tensoring with $\hatS$ over $S_\calL$, and that it satisfies $\hatCO = C \circ \CO^0_{\bL, \calL}$.  It remains to show that $C$ is an isomorphism, or equivalently that the natural map
\[
\HF^*_\calL(\bL, \bL) \to \HF^*_\calL(\bL, \bL) \otimes_{S_\calL} \hatS
\]
is an isomorphism.

To do this, we first use \cref{lemCommAlg}\ref{itmCompletion} again to interpret $\HF^*_\calL(\bL, \bL) \otimes_{S_\calL} \hatS$ as the $\m$-adic completion of $\HF^*_\calL(\bL, \bL)$.  It is then left to show that $\HF^*_\calL(\bL, \bL)$ is $\m$-adically complete.  For this, recall from the proof of \cref{lemJacisol} that $\HF^*_\calL(\bL, \bL) \cong \Jac_\calL W_L$, which is Artinian since $\calL$ is isolated.  We then have by \cite[Proposition 8.6]{AtiyahMacdonald} that the ideal $\m \Jac_\calL W_L$ in $\Jac_\calL W_L$ is nilpotent.  Thus $\Jac_\calL W_L$, and hence $\HF^*_\calL(\bL, \bL)$, is indeed $\m$-adically complete, as wanted.
\end{proof}

\begin{remark}
\label{rmkDropOpposite}
The above discussion shows that $\HF^*_\calL(\bL, \bL)$ and $\HF^*_{\hatS}(\bL, \bL)$ are both isomorphic to $\Jac_\calL W_L$, which is (graded-)commutative.  Therefore they both naturally coincide with their opposite algebras.
\end{remark}

\begin{remark}
One can use these methods to give an alternative proof of \cref{propAHH}, computing of $\HH^*(\calA)$.  Similarly, one can reprove \cref{thmGen}\ref{genEspace}, namely that $\OC_\lambda(\HH_*(\calA))$ lies in the generalised $\lambda$-eigenspace of $c_1 \qprod$, in the case where $L$ is a torus.
\end{remark}


\section{Worked examples}
\label{secExamples}


\subsection{The cubic surface}
\label{exCubic}

Let $X$ be the cubic surface, or equivalently the monotone $6$-point blowup of $\CC\PP^2$.  Our goal is to constrain the possible superpotentials of monotone tori in $X$ using \cref{corLocalFactors}, so our first task is to decompose $\QH^*(X)$ as a product of local rings.  This quantum cohomology ring was computed by Crauder--Miranda \cite{CrauderMiranda} and an explicit presentation is given by Sheridan in \cite[Proposition B.1]{SheridanFano}.

Viewing $X$ as $\operatorname{Bl}_6 \CC\PP^2$, $\QH^*(X)$ is generated as an algebra by the classes $H, E_1, \dots, E_6$ of a hyperplane and of the six exceptional spheres respectively.  A basis is given by $1, H, P, E_1, \dots, E_6$, where $P$ is the point class, and with respect to this basis quantum multiplication by $H$ and $E_1$ have matrices
{\tiny
\[
H \qprod = \begin{pmatrix} 0 & 120 & 252 & 42 & 42 & 42 & 42 & 42 & 42\\ 1 & 63 & 120 & 25 & 25 & 25 & 25 & 25 & 25\\ 0 & 1 & 0 & 0 & 0 & 0 & 0 & 0 & 0 \\ 0 & -25 & -42 & -15 & -9 & -9 & -9 & -9 & -9\\ 0 & -25 & -42 & -9 & -15 & -9 & -9 & -9 & -9\\ 0 & -25 & -42 & -9 & -9 & -15 & -9 & -9 & -9\\ 0 & -25 & -42 & -9 & -9 & -9 & -15 & -9 & -9\\ 0 & -25 & -42 & -9 & -9 & -9 & -9 & -15 & -9\\ 0 & -25 & -42 & -9 & -9 & -9 & -9 & -9 & -15 \end{pmatrix} \quad \text{\normalsize and} \quad E_1 \qprod = \begin{pmatrix} 0 & 42 & 84 & 20 & 14 & 14 & 14 & 14 & 14\\ 0 & 25 & 42 & 15 & 9 & 9 & 9 & 9 & 9\\ 0 & 0 & 0 & -1 & 0 & 0 & 0 & 0 & 0\\ 1 & -15 & -20 & -9 & -5 & -5 & -5 & -5 & -5\\ 0 & -9 & -14 & -5 & -5 & -3 & -3 & -3 & -3\\ 0 & -9 & -14 & -5 & -3 & -5 & -3 & -3 & -3\\ 0 & -9 & -14 & -5 & -3 & -3 & -5 & -3 & -3\\ 0 & -9 & -14 & -5 & -3 & -3 & -3 & -5 & -3\\ 0 & -9 & -14 & -5 & -3 & -3 & -3 & -3 & -5 \end{pmatrix}.
\]
}
One can compute (we used Macaulay2 for this) that these maps have a common eigenvector
\[
V = 48+21H+P-7(E_1+\dots+E_6),
\]
with eigenvalue $21$ and $7$ respectively.  By symmetry, $V$ is an eigenvector for every $E_i\qprod$, with eigenvalue $7$.  One can also compute that on the orthogonal complement of $V$ the elements $H+6$ and $E_i+2$ are nilpotent.  We conclude that if $\Char \kk \neq 3$ then there are two maximal ideals in $\QH^*(X)$, namely
\[
(H-21, E_1 - 7, \dots, E_6 - 7) \quad \text{and} \quad (H+6, E_1+2, \dots, E_6+2),
\]
whilst if $\Char \kk = 3$ then these two ideals coincide and are the unique maximal ideal in $\QH^*(X)$.  So in the former case $\QH^*(X)$ has two local factors, say $\calQ_1$ (spanned by $V$) and $\calQ_2$ (its orthogonal complement), whilst in the latter case $\QH^*(X)$ is already local and does not decompose further.  Let us restrict attention to orientable Lagrangians, so that the second part of \cref{rmkArtinian} applies and we can refer to the summands of the Fukaya category as $\Fuk(X)_{\calQ_1} = \Fuk(X)_{21, \calQ_1}$ and $\Fuk(X)_{\calQ_2} = \Fuk(X)_{-6, \calQ_2}$ outside characteristic $3$, and $\Fuk(X) = \Fuk(X)_{0, \QH^*(X)}$ in characteristic $3$.  Sheridan works over $\CC$ and refers to $\Fuk(X)_{\calQ_1}$ and $\Fuk(X)_{\calQ_2}$ as the small and large summands of the category respectively.

Now viewing $X$ as a cubic surface in $\CC\PP^3$, it can be degenerated to the \emph{Cayley cubic} $X_0$ defined by
\[
\{[w:x:y:z] \in \CC\PP^3 : xyz+ wyz + wxz + wxy = 0\}.
\]
$X_0$ has four ordinary double points, so the degeneration gives rise to four pairwise disjoint Lagrangian spheres $S_1, \dots, S_4$ in $X$ as vanishing cycles.  They are automatically monotone because they are simply-connected, and satisfy $\HF^*(S_i, S_j) = 0$ for $i \neq j$ since $S_i$ and $S_j$ are disjoint.  Moreover, degree considerations in the Oh spectral sequence $\rH^*(S_i) \implies \HF^*(S_i, S_i)$ show that $\HF^*(S_i, S_i) \neq 0$.  The $S_i$ thus constitute four orthogonal non-zero objects in the monotone Fukaya category.

Suppose $\Char \kk \neq 3$.  The existence of the $S_i$ means that at least one of $\Fuk(X)_{\calQ_1}$ and $\Fuk(X)_{\calQ_2}$ contains orthogonal non-zero objects, so in particular contains a non-zero object which doesn't split-generate it.  It cannot be $\Fuk(X)_{\calQ_1}$, since $\calQ_1$ is $1$-dimensional, so it must be $\Fuk(X)_{\calQ_2}$.  Thus, by \cref{TheoremD}, $\Fuk(X)_{\calQ_2}$ cannot be a toroidal subcategory in the sense of \cref{defToroidal}.  We deduce that in \cref{corLocalFactors} only the factor $\calQ_1$ can be hit.

Suppose instead that $\Char \kk = 3$.  Now \cref{TheoremD} tell us that that full Fukaya category cannot be toroidal, so in \cref{corLocalFactors} there are no factors to hit.

Combining these observations proves \cref{corCubic}.


\subsection{The quadric threefold}
\label{exQuadric}

Consider again the quadric threefold $X$ from \cref{exQuadricQH}.  We saw there that if $\Char \kk \neq 2$ then
\begin{equation}
\label{eqQuadricSplit}
\QH^*(X) \cong \Jac W_L \times \kk
\end{equation}
for a specific monotone torus $L$, and that $\Jac W_L$ is local if $\Char \kk = 3$ and isomorphic to $\kk^{\times 3}$ otherwise.  We conclude that for any monotone torus in $X$, its potential can have at most two isolated critical points in characteristic $3$ (one non-degenerate and the other with local Jacobian ring $\kk[y]/(y-1)^3$), and at most four isolated critical points, all of which must be non-degenerate, if $\Char \kk \neq 2, 3$.

If $\Char \kk = 2$ then $\QH^*(X)$ is local, isomorphic to $\kk[E] / (E^4)$, where $E$ is a generator of $\rH^4(X; \ZZ)$ and satisfies $E^2 = H$ (a hyperplane class, as before).  This is well-known, but can be computed by combining our previous work with knowledge of the classical cohomology ring
\[
\rH^*(X; \ZZ) = \ZZ[H, E] / (H^2 - 2E, E^2),
\]
as follows.  For degree reasons, in $\QH^*(X; \ZZ)$ we must have $H^2 = 2E$ and $E^2 = \kappa H$ for some $\kappa \in \ZZ$.  We obtain
\[
\QH^*(X; \QQ) = \QQ[H, E] / (H^2 - 2E, E^2 - \kappa H) \cong \QQ[H] / (H^4 - 4\kappa H),
\]
and then the $\kk = \QQ$ case of \cref{exQuadricQH} gives $4 \kappa = \nu = 4$, so $\kappa = 1$.  Therefore
\[
\QH^*(X; \ZZ) = \ZZ[H, E] / (H^2 - 2E, E^2 - H).
\]
So in characteristic $2$ we get $\kk[E] / (E^4)$, as claimed, and we see that the potential of a monotone torus can have at most one isolated critical point, with this algebra as its local Jacobian ring.

Going further, it was shown by Ivan Smith \cite[Lemma 4.6]{SmithQuadrics} over $\CC$, and by Evans--Lekili \cite[Section 7.1]{EvansLekiliGeneration} over arbitrary $\kk$, that the summand of the Fukaya category not split-generated by the torus $L$ is split-generated by the Lagrangian sphere $V$ appearing in \cref{corQuadric}.  If $\Char \kk \neq 2$ then this is the summand of the category corresponding to the $\kk$ factor in \eqref{eqQuadricSplit}, whilst if $\Char \kk = 2$ then it is the whole category (we'll assume that we only allow orientable Lagrangians so \cref{rmkArtinian} applies).  If this summand of the category had a toroidal generator $T$ then, by \cref{TheoremD}, $T$ would split-generate $V$, so in particular $T$ and $V$ would have to intersect.  So if we take a monotone torus disjoint from $V$ then in \cref{corLocalFactors} the factor of $\QH^*(X)$ corresponding to $V$ cannot be hit .

The upshot of these discussions is \cref{corQuadric}.


\subsection{The $4$-point blowup of $\CC\PP^2$}

Recall from \cref{ex4Blowup} that the monotone blowup $X = \operatorname{Bl}_4 \CC\PP^2$ of $\CC\PP^2$ at four points contains a monotone torus $L$ for which $\Jac W_L$ fills out a $3$-dimensional piece of the $7$-dimensional quantum cohomology algebra.  In \cite{Sanda} Sanda computed the $4$-dimensional complement to $\Jac W_L$ inside $\QH^*(X)$ using Lagrangian spheres, but we now explain how it's possible to use just tori, as follows.  This is somewhat orthogonal to the main ideas of the paper, but illustrates how one can sometimes combine critical points of different tori.

Pascaleff--Tonkonog \cite[Section 4]{PascaleffTonkonog} explain how $L$ can be \emph{mutated} along certain vectors in $\ZZ^2$ called \emph{mutation directions} to produce a new monotone torus whose superpotential is obtained from $W_L$ by a prescribed change of variables (the \emph{wall-crossing formula}).  Moreover, the mutation directions can themselves be mutated, and the new torus can be further mutated along the mutated directions.  This whole process can then be iterated as desired.  We will be interested in the mutation of $L$ along $(-1, 0)$, and the mutation of the new torus along the mutation of $(0, -1)$, which is just $(0, -1)$ itself.  We will similarly be interested in the mutation of $L$ along $(0, -1)$, and the mutation of the new torus along the mutation of $(-1, 0)$, which is $(-1, -1)$.  We use variables $(u, v)$, $(s, t)$, $(u', v')$, and $(s', t')$ for the potentials of these four tori respectively.  This is summarised, along with the corresponding changes of variable, in \cref{figMutation}.
\begin{figure}[ht]
\begin{tikzpicture}
\tikzset{blob/.style={circle, fill, inner sep=2pt}}
\def\a{-2.5}
\def\b{0}
\def\c{-5}
\def\d{0}
\draw (0, 0) node[blob, label=above:{$(x, y)$}]{};
\draw (\a, \b) node[blob, label=above:{$(u, v)$}]{};
\draw (\c, \d) node[blob, label=above:{$(s, t)$}]{};
\draw (-\a, \b) node[blob, label=above:{$(u', v')$}]{};
\draw (-\c, \d) node[blob, label=above:{$(s', t')$}]{};

\draw[postaction={decorate}, decoration={markings, mark=at position 0.5 with {\arrow{>}}}] (0,0) -- (\a, \b) node[pos=0.5, below]{$(-1, 0)$};
\draw[postaction={decorate}, decoration={markings, mark=at position 0.5 with {\arrow{>}}}] (\a, \b) -- (\c, \d) node[pos=0.5, below]{$(0, -1)$};
\draw[postaction={decorate}, decoration={markings, mark=at position 0.5 with {\arrow{>}}}] (0,0) -- (-\a, \b) node[pos=0.5, below]{$(0, -1)$};
\draw[postaction={decorate}, decoration={markings, mark=at position 0.5 with {\arrow{>}}}] (-\a, \b) -- (-\c,\d) node[pos=0.5, below]{$(-1, -1)$};
\end{tikzpicture}

{\footnotesize
\begin{gather*}
\lb s+t+st, \frac{t(1+s)}{s}\rb = \lb u(1+v), v\rb = \textcolor{blue}{\boldsymbol{(x, y)}} = \lb u', \frac{v'(1+u')}{u'}\rb = \lb s'+t', \frac{t'(1+s'+t')}{s'}\rb \\[1ex]
\lb s, \frac{t(1+s)}{s} \rb = \textcolor{blue}{\boldsymbol{(u, v)}} = \lb \frac{x}{1+y}, y\rb = \lb \frac{u'^2}{u'+v'+u'v'}, \frac{v'(1+u')}{u'}\rb = \lb \frac{s'}{1+t'}, \frac{t'(1+s'+t')}{s'} \rb \\[1ex]
\lb s+t+st, \frac{t(s+t+st)}{s(1+t)}\rb = \lb u(1+v), \frac{uv(1+v)}{1+u+uv} \rb = \lb x, \frac{xy}{1+x}\rb = \textcolor{blue}{\boldsymbol{(u', v')}} = \lb s'+t', \frac{t'(s'+t')}{s'}\rb \\[1ex]
\textcolor{blue}{\boldsymbol{(s, t)}} = \lb u, \frac{uv}{1+u} \rb = \lb \frac{x}{1+y}, \frac{xy}{1+x+y} \rb = \lb \frac{u'^2}{u'+v'+u'v'}, \frac{u'v'}{u'+v'}\rb = \lb \frac{s'}{1+t'}, t' \rb\\[1ex]
(s(1+t), t) = \lb \frac{u(1+u+uv)}{1+u}, \frac{uv}{1+u} \rb = \lb \frac{x(1+x)}{1+x+y}, \frac{xy}{1+x+y} \rb = \lb \frac{u'^2}{u'+v'}, \frac{u'v'}{u'+v'}\rb = \textcolor{blue}{\boldsymbol{(s', t')}}
\end{gather*}
}
\caption{Mutations and changes of variable.\label{figMutation}}
\end{figure}
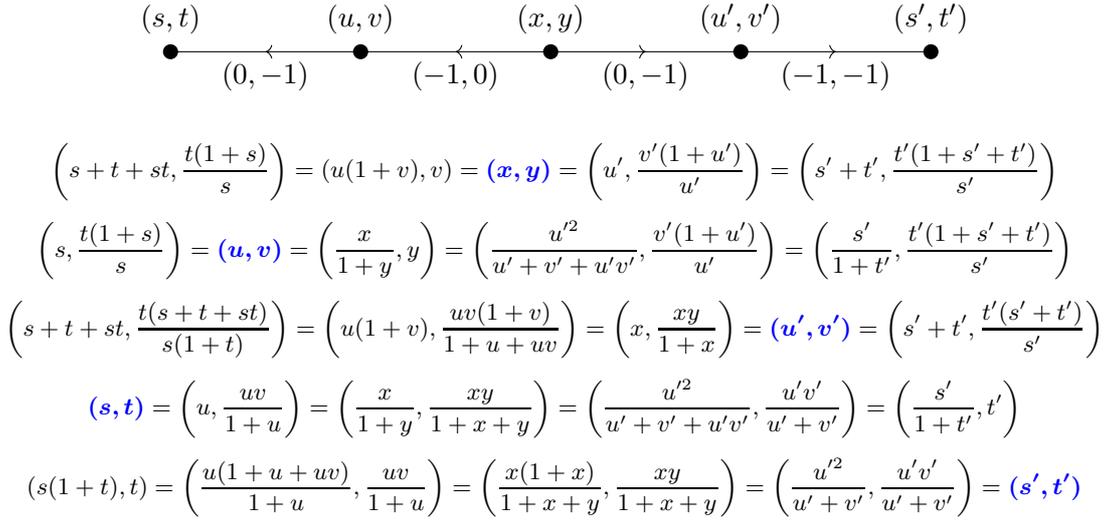

The superpotential of each of these mutated tori has a critical point not visible to $L$---in the sense that it cannot be obtained from a critical point of $W_L$ by the relevant change of variables---or to any of the other tori we're considering.  Explicitly, these new critical points are
\[
(u, v) = (-1, -1), \quad (u', v') = (-1, 1), \quad (s, t) = (-1, -1), \quad \text{and} \quad (s', t') = (1, -1).
\]
Each of these critical points gives rise to a factor of $\kk$ in $\QH^*(X)$.  Moreover, these are all orthogonal to each other and to the $3$-dimensional piece of $\QH^*(X)$ coming from $L$, for the following reason.  Let $\LL_1$ and $\LL_2$ be two distinct torus objects among those that we're considering.  It suffices, as in \cref{propIdempotentsOrthogonal}, show that $\HF^*(\LL_1, \LL_2) = 0$ whenever it makes sense (i.e.~whenever $\LL_1$ and $\LL_2$ lie in the same $\Fuk(X)_\lambda$).  The proof of the wall-crossing formula in \cite{PascaleffTonkonog} involves computing Floer cohomology $\HF^*_U(\LL_1, \LL_2)$ in a Liouville domain $U \subset X$, and because our critical points are not related by the changes of variables we have $\HF_U^*(\LL_1, \LL_2) = 0$.  But there is a spectral sequence from $\HF_U^*(\LL_1, \LL_2)$ to $\HF^*(\LL_1, \LL_2)$, which accounts for the contributions of holomorphic strips exiting $U$, so we conclude that $\HF^*(\LL_1, \LL_2) = 0$ as wanted.  We therefore have
\[
\QH^*(X) \cong \underbrace{\kk \times \kk[\xi]/(\xi^2-(\xi+1))}_L \times \underbrace{\kk \times \kk \times \kk \times \kk}_\text{mutated tori}.
\]

\bibliography{QHFukBib}
\bibliographystyle{siam}

\end{document}